\documentclass[11pt, reqno]{amsart}

\evensidemargin
\oddsidemargin
\makeatletter\@addtoreset{equation}{section}\makeatother

\usepackage{amsmath, amsthm, latexsym, amssymb, bbm}
\usepackage[dvips]{color}

\usepackage{calc,  subfigure, bm}
\usepackage{graphicx}

\newcommand{\e}{\varepsilon}

\renewcommand{\phi}{\varphi}

\newcommand{\ssup}[1] {{\scriptscriptstyle{({#1}})}}

\newcommand{\one}{{\mathbf 1}}

\newcommand{\R}{\mathbb R}

\newcommand{\N}{\mathbb N}

\newcommand{\E}{\mathbb E}
\renewcommand{\P}{\mathbb P}

\newtheorem{theorem}{Theorem}[section]

\newtheorem{lemma}[theorem]{Lemma}

\newtheorem{prop}[theorem]{Proposition}

\def\1{{\mathchoice {1\mskip-4mu\mathrm l}      
{1\mskip-4mu\mathrm l}
{1\mskip-4.5mu\mathrm l} {1\mskip-5mu\mathrm l}}}

\renewcommand{\subsection}{\secdef \subsct\sbsect}
\newcommand{\subsct}[2][default]{\refstepcounter{subsection}
\vspace{0.15cm}
{\flushleft\bf \arabic{section}.\arabic{subsection}~\bf #1  }
\nopagebreak\nopagebreak}
\newcommand{\sbsect}[1]{\vspace{0.1cm}\noindent
{\bf #1}\vspace{0.1cm}}

\newcounter{remnr}
\newenvironment{remark}{\refstepcounter{remnr}
{\sf Remark~\arabic{remnr}.\ }\nopagebreak  }%
{\nopagebreak 
\medskip}

\renewcommand{\phi}{\varphi}

\renewcommand{\P}{\mathbb{P}}
\renewcommand{\E}{\mathbb{E}}

\begin{document}

\title{Time-dependent balls and bins model with positive feedback}

\author[Nadia Sidorova]{}

\maketitle

\centerline{\sc Nadia Sidorova\footnote{Department of Mathematics, University College London, Gower Street, London WC1 E6BT, UK, {\tt n.sidorova@ucl.ac.uk}.
} }

\vspace{0.4cm}


%
\vspace{0.4cm}

\begin{quote}
{\small {\bf Abstract:} 
Balls and bins models are classical probabilistic models where balls are added to bins at random according to a certain rule.
The balls and bins model with feedback is a non-linear generalisation of the P\'olya urn, where the probability of a new ball choosing a bin with $m$ balls is proportional to $m^{\alpha}$, with $\alpha$ being the feedback parameter. It is known that if the feedback is positive (i.e.\ $\alpha>1$) then the model is monopolistic: there is a finite time after which one of the bins will receive all incoming balls. We consider a time-dependent version of this model, where 
$\sigma_n$ independent balls are added at time $n$ instead of just one. 
We show that if $\alpha>1$ then one of the bins gets all but a negligible number of balls, and identify a phase transition in the growth of $(\sigma_n)$ between the monopolistic and non-monopolistic behaviour. 
We also describe the critical regime, where the probability of monopoly is strictly between zero and one. Finally, we show that in the feedback-less case $\alpha=1$ no dominance occurs, that is, each bin gets a non-negligible proportion of balls eventually. This is in sharp contrast with a similar model where 
new balls added at time $n$ are all placed in the same bin 
rather than independently.

}
\end{quote}
\vspace{5ex}

{\small {\bf AMS Subject Classification:} Primary 60C05.
Secondary 68R05, 90B80, 60K35.

}
\vspace{4ex}



\bigskip

\section{Introduction}

\emph{Balls and bins models} are classical probabilistic models, which find numerous applications in Economics and Computer Science. They involve a number of bins, to which balls are added according to a certain rule. A simple example of a balls and bins model is equivalent to the P\'olya urn process introduced by Eggenberger and P\'olya in 1923, \cite{EP}. 
The model has two bins to which balls are added one by one, with the probability of a ball landing in a particular bin being proportional to the number of balls already in the bin. 
This is a prime example of a \emph{reinforced} random process, 
\cite{Ps}.
\smallskip

The balls and bins model \emph{with feedback}, or a \emph{non-linear P\'olya urn}, is a non-linear generalisation of the 
P\'olya-Eggenberger model, where the probability of a new ball choosing a bin with $m$ balls is proportional to $f(m)$, with some feedback function $f$. A commonly studied scenario is when there are two bins and $f(m)=m^{\alpha}$, with some positive exponent $\alpha$. 
The case $\alpha=1$ corresponds to the original P\'olya-Eggenberger model, while the cases $\alpha>1$ and $\alpha\in (0,1)$ are referred to as 
\emph{positive} and~\emph{negative} feedback, respectively. 
In the model with positive feedback the bin with a larger number of balls keeps getting more balls, reflecting the principle ``the strong gets stronger'', which leads to its eventual dominance. If the feedback is negative ``the strong gets weaker'', which pushes the bins towards an equilibrium.   
\smallskip

The balls and bins model with feedback 
was introduced in~\cite{DFM} and was motivated by 
the problems of economic competition. As noted in~\cite{SV}, the old industrial economies were more prone to a negative feedback, while 
the market dynamics of new information economy is typically governed by positive feedback. 
This is due to its network structure and to the fundamental economic effect: it is better to be connected to a bigger network rather than to a smaller one; or, in other words, the more popular a company is the more likely it is to get a new customer. 
This basic mechanism of the network evolution is called \emph{preferential attachment}, \cite{BA}, and the balls and bins process with feedback is a natural example of a preferential attachment process. 
\smallskip

It is well-known that if there are two bins and $\alpha=1$ then the proportion of the balls in each bin in the long run converges to a beta-distributed random variable (in particular, to a uniform distribution if both bins initially contain the same number of balls).  
\smallskip

%
%
%
%

For the positive feedback scenario $\alpha>1$, 
it was shown in~\cite{DFM} that 
the proportion of balls in each bin converges almost surely to a $\{0,1\}$-valued random variable, 
an event that we call \emph{dominance}. 
A much stronger result was obtained in~\cite{KK}, where an equivalent model was considered in the context of neuron growth.
It was proved that almost surely 
one of the  bins gets all but finitely many balls, an event that we call
\emph{monopoly}. The onset of monopoly was characterised in more detail in~\cite{O2}. A scaling limit for the probability of a certain bin to win, given that the initial number of balls is large, was described in~\cite{MOS}. A generalised model with a growing number of bins
was studied in~\cite{CCL}. 
\smallskip

If the feedback is negative the system tends to converge towards an equal number of balls in each bin, 
and neither monopoly nor dominance occurs, see~\cite{DFM, KK, O}.
\smallskip

We consider a \emph{time-dependent} balls and bins 
model with feedback, where the number of balls added at time $n$
is no longer one but is a function of $n$. 
The balls added at time $n$ choose between the bins independently, 
and for each ball the probability to land in a bin with $m$ balls is proportional to $m^{\alpha}$ --- just like in the original model. 
We assume that there are two bins, and that the feedback is either positive ($\alpha>1$) or non-existent ($\alpha=1$).
The aim of the paper is to analyse dominance and monopoly in that time-inhomogeneous scenario. This requires a new approach as methods used in~\cite{CCL, KK} rely on Rubin's construction, which fails if the number of added balls is not equal to one.
\smallskip

Let $(\sigma_n)$ 
be a positive sequence representing the number
of added balls at times $n\in\N$. 
Denote by $\tau_0$ the initial number of balls and, for each $n$, let
\begin{align*}
\tau_n=\tau_0+\sum_{i=1}^n \sigma_i
\end{align*}
be the total number of balls at time $n$.
\smallskip

Denote by $0<T_0<\tau_0$ the initial deterministic number of the balls in the first bin. Given that the bin contains $T_{n}$ 
balls at time $n$, we denote by 
\begin{align}
\label{def1th}
\Theta_n=\frac{T_n}{\tau_n}
\end{align}
the proportion of balls in the first bin and 
define
\begin{align}
\label{def1t}
T_{n+1}=T_{n}+B_{n+1},
\end{align}
where $B_{n+1}$ is a Binomial random variable with size $\sigma_{n+1}$ and parameter 
\begin{align}
\label{defp5}
P_n=\frac{\Theta_n^{\alpha}}{\Theta_n^{\alpha}+(1-\Theta_n)^{\alpha}},
\end{align}
otherwise independent of $\mathcal{F}_n=\sigma(B_1,\dots,B_n)$. 
Denote the corresponding probability and expectation by 
$\P$ and $\E$, and the conditional probabilities and expectations by $\P_{\mathcal{F}_n}$ and $\E_{\mathcal{F}_n}$, respectively. 
\smallskip

We denote by
\begin{align*}
\mathcal{D}=\big\{\exists\lim_{n\to\infty}\Theta_n\in \{0,1\}\big\}
\end{align*}
the event that eventually the number of balls in one of the bins is negligible, and call this event \emph{dominance}. Further, we denote by
\begin{align*}
\mathcal{M}=\big\{B_n=0\text{ eventually for all }n\big\}
\cup
\big\{B_n=\sigma_n\text{ eventually for all }n\big\}
\end{align*}
the event that eventually all balls are added to one of the bins, and call this event \emph{monopoly}. It is easy to see that 
\begin{align*}
\mathcal{M}\subset \mathcal{D}.
\end{align*}

\bigskip

Our first result is about the no feedback scenario. 

\begin{theorem}  
\label{polya}
Suppose $\alpha=1$. Then 
$\Theta_n$ converges almost surely to a random variable $\Theta$, 
and $\P(\mathcal{D})=0$. 
\end{theorem}
\smallskip

\begin{remark} Convergence of $\Theta_n$ follows from a simple martingale argument so the non-trivial part of this theorem is to show that almost surely there is no dominance. 
\end{remark}

\begin{remark}
A similar problem 
was considered in~\cite{P, S}, where all $\sigma_n$ new balls were added 
to the same bin in a bulk (rather than independently to both bins). 
This is equivalent to replacing the $(\Theta_{n-1},\sigma_n)$-binomial random variables $B_n$ by $\sigma_n I_n$, where $I_n$ are Bernoulli random variables with parameters $\Theta_{n-1}$, otherwise independent of the past.
It was shown in~\cite{P} that
$$\sum\limits_{n=0}^{\infty}\Big(\frac{\sigma_{n+1}}{\tau_n}\Big)^2=\infty
\qquad\Leftrightarrow\qquad
\P(\mathcal{D})=1.
$$ 
Further, it was proved in~\cite{S} that even if the above series converges, we have 
\begin{align*}
\sum_{n=1}^{\infty}\frac{1}{\tau_n}<\infty 
\qquad\Leftrightarrow\qquad
\P(\mathcal{D})>0,
\end{align*}
provided that $(\sigma_n)$ satisfies some regularity conditions.
In other words, there is a phase transition from dominance 
to no dominance 
depending on the growth rate of $(\sigma_n)$.
This is in a striking contrast with our result, where dominance does not occur for any $(\sigma_n)$. 
This effect can be heuristically explained by higher step-by-step fluctuations of the second model, i.e.\ $\text{var} (B_n)=\sigma_n\Theta_{n-1}(1-\Theta_{n-1})\ll \sigma_n^2\Theta_{n-1}(1-\Theta_{n-1})=\text{var} (\sigma_nI_n)$, leading to more extreme behaviour of $(\Theta_n)$ for faster-growing $(\sigma_n)$.  
\end{remark}

Now we turn our attention to the positive feedback scenario and assume that 
$\alpha>1$. 
Denote 
\begin{align*}
\rho_n=\frac{\sigma_{n+1}}{\tau_n},\quad n\in\N_0.
\end{align*}

We will often (but not always) impose the following regularity conditions on the sequence $(\sigma_n)$:
\begin{itemize}
\item[(S)] $(\sigma_n)$ is either bounded or tends to infinity;
\item[(R)] $(\rho_n)$ is either bounded or tends to infinity. 
\end{itemize}


\begin{theorem}  
\label{main0}
Suppose $\alpha>1$, and 
{\rm (S)} and {\rm (R)} are satisfied. Then $\P(\mathcal{D})=1$.
\end{theorem}

This theorem means that a time-dependent balls and bins model with positive feedback always exhibits dominance regardless of the growth of $(\sigma_n)$. Monopoly is more delicate, and whether or not it occurs is determined by the growth parameter 
\begin{align*}
\theta=\lim_{n\to\infty}\alpha^{-n}\log\tau_n\in [0,\infty].
\end{align*}
In the sequel we assume that this limit exists. We will distinguish between three regimes: supercritical ($\theta=\infty$), subcritical ($\theta=0$), and critical ($0<\theta<\infty$). We will see that monopoly occurs with probability zero and one in the  supercritical and subcritical regimes, respectively. In the critical regime the probability of monopoly depends on the finer details of the growth of $(\tau_n)$
and can be strictly between zero and one.

\begin{theorem}[Supercritical regime]
\label{main2}
Suppose $\alpha>1$ and $\theta=\infty$. Then $\P(\mathcal{M})=0$.
\end{theorem}


\begin{theorem}[Subcritical regime]
\label{main1}
Suppose $\alpha>1$, $\theta=0$, and {\rm (S)} is satisfied. 
\smallskip

\begin{itemize}

\item If $(\rho_n)$ is bounded then  $\P(\mathcal{M})=1$.
\smallskip

\item If $\rho_n\to\infty$ then
\begin{align*}
\P(\mathcal{M})=
\left\{\begin{array}{ll} 1 & \text{ \rm if }\lambda<1,\\
0 & \text{ \rm if }\lambda>1,\end{array}\right.
\end{align*}
where 
\begin{align*}
\lambda=
\limsup_{n\to\infty}
\frac{\sigma_{n+1}\sigma_{n-1}^{\alpha}}{\sigma_n^{\alpha+1}}.
\end{align*}
In particular, if $\lambda$ is  a limit (finite or infinite) then $\lambda=0$ and $\P(\mathcal{M})=1$.
\end{itemize}
\end{theorem}
\medskip

\begin{remark} 
In the subcritical regime one has monopoly almost surely unless the sequence $(\sigma_n)$ is rather irregular. 
\end{remark}

\begin{remark} It is easy to see that if $(\rho_n)$ is bounded then $\theta=0$, see Lemma~\ref{omegazero}, so there is no need to assume the latter. We do it only to make the distinction between the regimes according to the value of $\theta$ more transparent.
\end{remark}

%
%
%

\begin{theorem}[Critical regime]
\label{main3}
Suppose $\alpha>1$ and $\theta\in (0,\infty)$. Then 
\begin{align*}
\sum_{n=0}^{\infty}\frac{\tau_{n+1}}{\tau_n^{\alpha}}=\infty
&\qquad\Rightarrow \qquad
\P(\mathcal{M})=0,\\
\sum_{n=0}^{\infty}\frac{\tau_{n+1}}{\tau_n^{\alpha}}<\infty
&\qquad\Rightarrow \qquad
\P(\mathcal{M})\in (0,1).
\end{align*}
%
%
\end{theorem}
\smallskip


\begin{remark} Theorems~\ref{main2} and~\ref{main1} suggest that the phase transition from monopoly to no monopoly happens 
 when $(\tau_n)$ changes from growing slowly to growing fast. Surprisingly, this is no longer true in the critical regime. 
 For example, the sequence   $\tau_n=\lfloor b^n e^{\alpha^n}\rfloor$ exhibits monopoly with positive probability if $b>1$ and almost surely no monopoly if $b\le 1$ .
\end{remark}

The heuristics of the positive feedback scenario can be described as follows. 
By~\eqref{defp5} we have for all $n\in\N_0$
\begin{align}
\label{defp}
P_n=\psi(\Theta_n),
\end{align}
where
\begin{align*}
\psi(x)=\frac{x^{\alpha}}{x^{\alpha}+(1-x)^{\alpha}},\qquad x\in [0,1].
\end{align*}
It follows from~\eqref{def1th} and~\eqref{def1t} that 
\begin{align}
\label{itit11}
\Theta_{n+1}=\frac{\tau_n}{\tau_{n+1}}\Theta_n+\frac{1}{\tau_{n+1}}B_{n+1},
\end{align}
and therefore by~\eqref{defp} we have approximately 
\begin{align}
\label{itit111}
\Theta_{n+1}
\approx\frac{\tau_n}{\tau_{n+1}}\Theta_n+\frac{\sigma_{n+1}}{\tau_{n+1}}P_{n}
=\frac{\tau_n}{\tau_{n+1}}\Theta_n+\frac{\sigma_{n+1}}{\tau_{n+1}}\psi(\Theta_n)
=\psi_n(\Theta_n),
\end{align}
where $\psi_n$ is defined as a convex combination of the identity function and the function $\psi$ with the coefficients $\frac{\tau_n}{\tau_{n+1}}$ and $\frac{\sigma_{n+1}}{\tau_{n+1}}$, respectively. Hence the graph of $\psi_n$ is squeezed in the grey area between those two functions  on Picture~\ref{figure}, and each $\psi_n$ has three fixed points $0$, $1/2$, and $1$. Since $\psi_n'(1/2)>1$ the equilibrium $1/2$ is not an attractor, so the only natural candidates for the limit of $(\Theta_n)$ are the end points $0$ and $1$. This makes the conjecture of dominance plausible, unless the fluctuations of $(B_n)$ overpower the bias caused by the feedback, especially around the equilibrium.
\smallskip 

\begin{figure}
\label{figure}
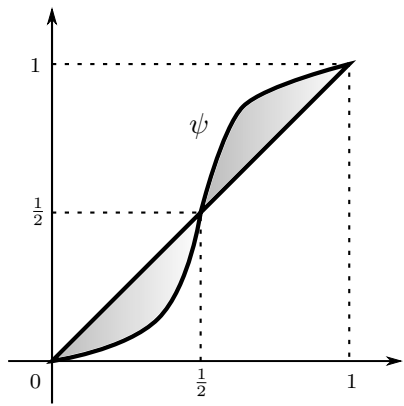
\caption{The graph of $\psi$ and the location of $\psi_n$}
\end{figure}

In order to understand the phase transition from the monopolistic to non-monopolistic behaviour, due to domination and symmetry
it suffices to consider the event $\{\Theta_n\to 0\}$. In that case we have $\psi(\Theta_n)\approx \Theta_n^{\alpha}$, and the iteration~\eqref{itit111}
turns into
\begin{align}
\label{itit222}
\Theta_{n+1}
\approx\frac{\tau_n}{\tau_{n+1}}\Theta_n+\frac{\sigma_{n+1}}{\tau_{n+1}}\Theta_n^{\alpha}.
\end{align}
Next we make a natural assumption that 
one of terms of~\eqref{itit222} is more important than the other in the long run. If the second term is negligible then we have
\begin{align}
\label{ass1}
\Theta_{n}\approx \frac{\tau_{n-1}}{\tau_n}\Theta_{n-1}
\approx\frac{\tau_{n-2}}{\tau_n}\Theta_{n-2}\approx \cdots\approx 
\frac{1}{\tau_n}.
\end{align}
Otherwise, if the first term is negligible, then $\sigma_{n+1}\approx\tau_{n+1}$ due to $\Theta_n\gg \Theta_n^{\alpha}$, and we obtain 
\begin{align*}
\Theta_{n}\approx \Theta_{n-1}^{\alpha}\approx \Theta_{n-2}^{\alpha^2}
\approx \cdots\approx e^{\alpha^n}.
\end{align*}
Equating 
\begin{align*}
\frac{1}{\tau_n}\approx e^{\alpha^n}
\end{align*}
we observe that the phase transition is likely to occur when 
\begin{align*}
\alpha^n \log\tau_n\approx\text{const},
\end{align*}
which indeed corresponds to our critical regime. 
\smallskip

This analysis suggests, in particular, that non-monopolistic regimes should be much easier to analyse than monopolistic ones. Indeed, to show that no monopoly occurs we need to bound $\Theta_n$ from below, which can be done by simply dropping the negligible part of~\eqref{itit222}.  
On the contrary, 
upper bounds for $\Theta_n$ required for monopolistic regimes prove 
to be quite challenging. Another difficulty we face is dealing with random fluctuations of $B_n$, which can be quite big and may begin to interfere  with the averages discussed so far. This may be particularly damaging around the equilibrium, where the averages have less power due to it being a fixed point of each $\psi_n$. 
\smallskip

%
%
%
%




The paper is organised as follows. 
In Section~\ref{s:nomo} we find a sufficient condition for non-occurrence of monopoly (Lemma~\ref{diverges}). This allows us to handle the supercritical case in the positive feedback scenario (Theorem~\ref{main2}). We also show that almost surely monopoly does not occur if there is no feedback (Lemma~\ref{l:lll}).
In Section~\ref{s:nofeedback} we consider the feedback-less case in detail
and strengthen the latter result from no monopoly  to no dominance 
(Theorem~\ref{polya}).  Then we move to the positive 
feedback scenario. In Section~\ref{s:away} we show
that the proportion $\Theta_n$ does not get stuck at the equilibrium and deviates from it significantly (although still infinitesimally) infinitely often. In Section~\ref{s:dom} we prove almost sure dominance (Theorem~\ref{main0}) by first deviating far enough from the equilibrium and then showing that the feedback drags $\Theta_n$ even further away. Section~\ref{s:pro} contains 
some technical lemmas  required later for the subcritical regime. 
In Section~\ref{s:sub} we prove monopoly in the two cases corresponding to a bounded $(\rho_n)$ and $(\rho_n)$ tending to infinity with 
$\lambda<1$ (Propositions~\ref{p:mod1} and \ref{p:fast}, respectively). In the third case of $(\rho_n)$ tending to infinity with 
$\lambda>1$ we use the same idea as in Section~\ref{s:nomo} to show no monopoly. This completes the proof of Theorem~\ref{main1}. Finally, Section~\ref{s:cri} is devoted to the critical regime. There we prove that the probability of monopoly is always less then one, and provide a sufficient condition for it to be zero (Proposition~\ref{p:cri1}). Then we show that that condition is also a necessary one (Proposition~\ref{p:cri2}), completing the proof of Theorem~\ref{main3}.
\smallskip

Finally, we would like to introduce notation that will be used throughout the paper. 
Let
\begin{align*}
\e_n=\frac{B_n-\sigma_{n}P_{n-1}}{\sqrt{\sigma_{n}P_{n-1}(1-P_{n-1})}}, \quad n\in \N,
\end{align*}
be the normalised fluctuation of $B_n$.
It is easy to see that 
\begin{align}
\label{norm}
\E_{\mathcal{F}_{n-1}}\e_{n}=0
\qquad\text{and}\qquad
\E_{\mathcal{F}_{n-1}}\e_{n}^2=1.
\end{align}

It follows from~\eqref{def1t} and~\eqref{defp} that for all $n\in\N_0$
\begin{align}
\label{iteration1}
T_{n+1}=T_n+\sigma_{n+1}\psi(\Theta_n)+\e_{n+1}\sqrt{\sigma_{n+1}P_n(1-P_n)}.
\end{align}
At last, the following obvious  bound on $\psi$ will be used repeatedly:
\begin{align}
\label{psibounds}
x^{\alpha}\le \psi(x)\le 2^{\alpha-1}x^{\alpha}
\qquad\text{for all }x\in [0,1].
\end{align}

\bigskip


\section{No monopoly}
\label{s:nomo}

This section is devoted to non-monopolistic regimes. As we mentioned earlier, showing no monopoly is not hard, and a sufficient condition for non-occurrence of monopoly is established in the lemma below. In fact, 
it will follow from further results that this condition is a necessary one provided the sequence $(\sigma_n)$ satisfies the regularity conditions (S) and (R) and has a well-defined parameter $\theta$.

\begin{lemma} 
\label{diverges}
If 
\begin{align}
\label{div7}
\sum_{n=0}^{\infty}\frac{\sigma_{n+1}}{\tau_n^{\alpha}}=\infty,
\end{align}
then $\P(\mathcal{M})=0$. 
\end{lemma}

\begin{proof} 
First, let us show that 
\begin{align}
\label{BC}
\big\{T_n\to\infty\big\}\supset\Big\{\sum_{n=1}^{\infty}\sigma_nP_{n-1}=\infty\Big\}.
\end{align}
Observe that for all $n\in\N_0$ 
\begin{align}
\label{defm}
T_n=T_0+\sum_{i=1}^{n}B_i=T_0+M_n+Y_n,
\end{align}
where 
\begin{align*}
M_n=\sum_{i=1}^{n}\big(B_i-\sigma_iP_{i-1})
\qquad\text{and}\qquad
Y_n=\sum_{i=1}^{n}\sigma_iP_{i-1}.
\end{align*}
We have
\begin{align}
\langle M\rangle_n
=\sum_{i=1}^{n}\sigma_i P_{i-1}(1-P_{i-1})\le Y_n.
\label{my}
\end{align}

Suppose $Y_n\to\infty$. If $\langle M\rangle_n$ converges then $M_n$ converges, see~\cite[\S 12.13]{DW}, and $T_n\to\infty$ by~\eqref{defm}.
If $\langle M\rangle_n\to\infty$ then $M_n/\langle M\rangle_n\to 0$, see~\cite[\S 12.14]{DW}. By~\eqref{my} we have 
$M_n/Y_n\to 0$, which implies $T_n\to\infty$ by~\eqref{defm}. This completes the proof of~\eqref{BC}.
\smallskip

Second, observe that by~\eqref{psibounds} almost surely
\begin{align*}
\sum_{n=1}^{\infty}\sigma_{n}P_{n-1}
=\sum_{n=1}^{\infty}\sigma_{n}\psi(\Theta_{n-1})
\ge \sum_{n=1}^{\infty}\sigma_{n}\Theta_{n-1}^{\alpha}
=\sum_{n=1}^{\infty}\frac{\sigma_{n}}{\tau_{n-1}^{\alpha}}T_{n-1}^{\alpha}
\ge \sum_{n=1}^{\infty}\frac{\sigma_{n}}{\tau_{n-1}^{\alpha}}=\infty.
\end{align*}
Hence by~\eqref{BC} we have $\P(T_n\to\infty)=1$. By symmetry this implies $\P(\mathcal{M})=0$.
\end{proof}
\smallskip

\begin{proof}[Proof of Theorem~\ref{main2}] 
By Lemma~\ref{diverges} it suffices to show that $\theta=\infty$ implies~\eqref{div7}. 
Suppose the series converges. Then $\frac{\sigma_{n}}{\tau_{n-1}^{\alpha}}\to 0$. Let $k$ be such that $\sigma_{n}<\tau_{n-1}^{\alpha}$ for all $n\ge k$. Then for all $n\ge k$
\begin{align*}
\tau_n\le \tau_{n-1}+\tau_{n-1}^{\alpha}\le 2\tau_{n-1}^{\alpha}\le\cdots\le 2^{1+\alpha+\cdots+\alpha^{n-k-1}}\tau_k^{\alpha^{n-k}}
\le (2\tau_k)^{\alpha^{n-k}}.
\end{align*}  
This implies 
\begin{align*}
\theta\le \alpha^{-k}\log(2\tau_k)<\infty
\end{align*}
leading to a contradiction.
\end{proof}
\smallskip

\begin{lemma} 
\label{l:lll}
Suppose $\alpha=1$. Then $\P(\mathcal{M})=0$. 
\end{lemma}

\begin{proof} 
This follows from Lemmas~\ref{diverges} and Lemma~\ref{1101series} below.
\end{proof}
\smallskip

\begin{lemma} 
\label{1101series}
We have
\begin{align*}
\sum_{n=0}^{\infty}\frac{\sigma_n}{\tau_n}=\infty.
\end{align*}
\end{lemma}

\begin{proof} 
We have 
\begin{align*}
\tau_n=\tau_{n-1}\Big(1-\frac{\sigma_{n}}{\tau_{n}}\Big)^{-1}=\tau_0\prod_{k=1}^{n}\Big(1-\frac{\sigma_k}{\tau_k}\Big)^{-1}
=\tau_0\exp\Big\{-\sum_{k=1}^n\log \Big(1-\frac{\sigma_k}{\tau_k}\Big)\Big\}. 
\end{align*}
This implies 
\begin{align*}
-\sum_{k=1}^n\log \Big(1-\frac{\sigma_k}{\tau_k}\Big)=\log\frac{\tau_n}{\tau_0}\to \infty,
\end{align*}
which is equivalent to the statement of the lemma.
\end{proof}

\bigskip


\section{No feedback scenario}

\label{s:nofeedback}

In this section we use Laplace transforms to prove that no dominance occurs in the case $\alpha=1$.

\begin{proof}[Proof of Theorem~\ref{polya}] 
Observe that $(\Theta_n)$ is a bounded martingale as for all $n$
we have by~\eqref{itit11}
\begin{align*}
\E_{\mathcal{F}_{n-1}}\Theta_n=\frac{\tau_{n-1}}{\tau_n}\Theta_{n-1}+\frac{\sigma_n}{\tau_n}\psi(\Theta_{n-1})=\Theta_{n-1}. 
\end{align*}
Hence it converges to a random variable $\Theta$ bounded between zero and one. By symmetry it suffices to prove that $\P(\Theta=0)=0$. 
\smallskip

Denote by
\begin{align*}
f_n(\lambda)=\E e^{-\lambda \Theta_n}
\qquad\text{and}\qquad 
f(\lambda)=\E e^{-\lambda \Theta},
\qquad \lambda\in\R,
\end{align*} 
the Laplace transforms of $\Theta_n$ and $\Theta$, respectively. 
Since 
\begin{align*}
f(\lambda)\ge \P(\Theta=0)
\end{align*}
for all $\lambda$, it suffices to show that there is a sequence $(\lambda_m)$ such that 
\begin{align}
\label{bbb2}
\lim_{m\to\infty}f(\lambda_m)=0.
\end{align}
To do so, observe that there is $c\in (0,1)$ such that 
\begin{align}
\label{exp0}
e^{-x}\le 1-x+\frac{x^2}{2}
\end{align}
for all $x\in [0,c]$, and then define $\lambda_m=c\tau_m$.
\smallskip

Let us prove by induction over $k$ that  
\begin{align}
\label{indu}
f_n(\lambda_m)\le f_{n-k}\Big(\lambda_m-\lambda_m^2\sum_{i=n-k+1}^{n}\frac{\sigma_i}{\tau_i^2}\Big)
\end{align}
for all $m$, $n>m$, and $1\le k\le n-m$. 
\smallskip

First, observe that  for all $\lambda\in [0,c\tau_m]$ we have by~\eqref{itit11} and~\eqref{exp0} that
\begin{align*}
f_n(\lambda)
&=\E \Big[\exp\Big\{-\frac{\tau_{n-1}}{\tau_n}\lambda\Theta_{n-1}+\sigma_n\log
\Big(1-\Theta_{n-1}+\Theta_{n-1} \exp\Big\{-\frac{\lambda}{\tau_n}\Big\}\Big)\Big]\\
&\le \E \Big[\exp\Big\{-\frac{\tau_{n-1}}{\tau_n}\lambda\Theta_{n-1}
+\sigma_n\log\Big(1- \frac{\lambda}{\tau_n}\Theta_{n-1}+ \frac{\lambda^2}{2\tau_n^2}\Theta_{n-1}\Big)\Big\}\Big]
\end{align*}
for all $m$ and $n>m$ since 
\begin{align*}
\frac{\lambda}{\tau_n}\le \frac{\lambda}{\tau_m}=c.
\end{align*}
Further, using $\log(1+x)\le x$ and dropping the factor $1/2$  we obtain
\begin{align}
f_n(\lambda)
&\le \E \Big[\exp\Big\{-\lambda\Theta_{n-1}
+ \frac{\sigma_n}{\tau_n^2}\lambda^2\Theta_{n-1}\Big\}\Big]
\le f_{n-1}\Big(\lambda-\frac{\sigma_n}{\tau_n^2}\lambda^2\Big).
\label{it7}
\end{align}
This proves the statement~\eqref{indu} for $k=1$. 
Further, suppose~\eqref{indu} it is true for some $k$. Observe that 
\begin{align}
\label{bbb1}
\lambda_m-\lambda_m^2\sum_{i=n-k+1}^{n}\frac{\sigma_i}{\tau_i^2}\le \lambda_m=c\tau_m
\end{align}
and, since $\tau_m\le \tau_{n-k}$ and $c<1$,
\begin{align}
\label{bbb}
\lambda_m-\lambda_m^2\sum_{i=n-k+1}^{n}\frac{\sigma_i}{\tau_i^2}\ge
\lambda_m\Big[1-c\tau_m\int_{\tau_{n-k}}^{\infty}\frac{dx}{x^2}\Big]
=\lambda_m\Big[1-c\frac{\tau_m}{\tau_{n-k}}\Big]\ge \lambda_m(1-c)\ge 0.
\end{align}
By~\eqref{bbb1} and~\eqref{bbb} we can use~\eqref{it7} together with monotonicity of $f_{n-k-1}$ to obtain 
\begin{align*}
 f_{n-k}\Big(\lambda_m-\lambda_m^2\sum_{i=n-k+1}^{n}\frac{\sigma_i}{\tau_i^2}\Big)
& \le f_{n-k-1}\Big(\lambda_m-\lambda_m^2\sum_{i=n-k+1}^{n}\frac{\sigma_i}{\tau_i^2}
 -\frac{\sigma_{n-k}}{\tau_{n-k}^2}\Big(\lambda_m-\lambda_m^2\sum_{i=n-k+1}^{n}\frac{\sigma_i}{\tau_i^2}\Big)^2\Big)\\
& \le f_{n-k-1}\Big(\lambda_m-\lambda_m^2\sum_{i=n-k}^{n}\frac{\sigma_i}{\tau_i^2}\Big)
\end{align*}
by replacing the squared term by $\lambda_m^2$ only. 
This, together with the induction hypothesis~\eqref{indu} for $k$, completes the induction step from $k$ to $k+1$.  
\smallskip

Substituting $k=n-m$ into~\eqref{indu} and using~\eqref{bbb} together with monotonicity of $f_m$ we obtain
\begin{align}
f_n(\lambda_m)\le f_{m}\Big(\lambda_m-\lambda_m^2\sum_{i=m+1}^{n}\frac{\sigma_i}{\tau_i^2}\Big)
\le f_m \big(\lambda_m(1-c)\big)
=\E e^{-c(1-c)\tau_m\Theta_m}=\E e^{-c(1-c)T_m}
\label{lili7}
\end{align}
for all $m$ and $n>m$. 
\smallskip

By the dominated convergence theorem we have  $f_n(\lambda)\to f(\lambda)$ 
for all $\lambda>0$.  Hence we can take the limit in~\eqref{lili7} to obtain 
\begin{align*}
f(\lambda_m)\le \E e^{-c(1-c)T_m}
\end{align*}
for all $m$. 
To prove~\eqref{bbb2},
it remains to notice that 
\begin{align*}
\lim_{m\to\infty}\E e^{-c(1-c)T_m}= 0
\end{align*}
by the dominated convergence theorem as there is no monopoly, that is, $T_m\to\infty$ almost surely by Lemma~\ref{l:lll}.
\end{proof} 
\bigskip



\section{Getting away from the equilibrium} 

In this section we show that in the positive feedback scenario $\Theta_n$ deviates from the equilibrium far enough infinitely often.  

\label{s:away}

\begin{prop}  
\label{devp}
Suppose $\alpha>1$ and  {\rm (S)} is satisfied. 
Let $(\delta_n)$ be a positive sequence converging to zero and such that 
\begin{align}
\label{ser1}
\sum_{n=0}^{\infty}\delta_n\, \frac{\sigma_{n+1}}{\tau_{n+1}}<\infty. 
\end{align}
Then
\begin{align*}
\P\big(|\Theta_n-1/2|>\delta_n \text{ infinitely often}\big)=1.
\end{align*}
\end{prop}

The idea of the proof can be vaguely described as follows. 
We assume that $\Theta_n$ does not deviate from the equilibrium enough and aim to show that it is very unlikely. If $\Theta_n\approx 1/2$ then 
$P_n\approx 1/2$ and $\psi(\Theta_n)\approx \alpha (\Theta_n-1/2)+1/2$ as $\psi'(1/2)=\alpha$. 
Substituting this into~\eqref{iteration1} we obtain 
\begin{align*}
\Theta_{n}-\frac 1 2 
&\approx\Big(\frac{\tau_{n-1}}{\tau_{n}} +\alpha\frac{\sigma_n}{\tau_n}\Big)\Big(\Theta_{n-1}-\frac 1 2 \Big)+\e_{n}\frac{\sqrt{\sigma_{n}}}{2\tau_{n}}.  
\end{align*}
Carefully iterating this from a fixed large $m$ to $n\to\infty$, we obtain 
\begin{align}
\label{app1}
\Theta_{n}-\frac 1 2 \approx \pi_{m,n}\Big(\Theta_{m}-\frac 1 2 
+\frac{1}{2}\mu_{m,n}N_{m,n}\Big),
\end{align}
where $\pi_{m,n}$ and $\mu_{m,n}$ are some deterministic scales, and 
$N_{m,n}$ is a random variable arisen from the noises $(\e_i)$
between the times $m$ and $n$. Then we observe by a CLT argument 
that the distribution of $N_{m,n}$ is asymptotically close to normal. 
Finally, we observe that $\pi_{m,n}\mu_{m,n}\to \infty$ which, together with $\Theta_n\approx 1/2$, $\Theta_m\approx 1/2$ and $N_{m,n}$ being of a finite non-negligible order, makes~\eqref{app1} impossible.

\begin{proof}[Proof of Proposition~\ref{devp}]  
Denote 
\begin{align*}
\mathcal{H}_m=\big\{|\Theta_n-1/2|\le \delta_n\text{ for all }n\ge m\big\}.
\end{align*}
Since the events $\mathcal{H}_m$ are increasing 
it suffices to show that 
\begin{align*}
\P(\mathcal{H}_m)\to 0.
\end{align*}
%

By the mean value theorem we have 
\begin{align*}
\psi(x)=\psi'(\xi_x)\big(x-\frac{1}{2}\big)+\frac 1 2,
\end{align*}
where $\xi_x$ lies between $x$ and $1/2$. Hence 
\begin{align}
\label{ka1p}
\frac{\tau_n}{\tau_{n+1}}x+\frac{\sigma_{n+1}}{\tau_{n+1}}\psi(x)-\frac 1 2
&=\frac{\tau_n}{\tau_{n+1}} x+\frac{\sigma_{n+1}}{\tau_{n+1}}
\Big[\psi'(\xi_x)\big(x-\frac{1}{2}\big)+\frac 1 2\Big]-\frac 1 2
=\kappa_n(x) \big(x-\frac 1 2\big), 
\end{align}
where 
\begin{align*}
\kappa_n(x)=\frac{\tau_n}{\tau_{n+1}}+\frac{\sigma_{n+1}}{\tau_{n+1}}\psi'(\xi_x).
\end{align*}

It follows from~\eqref{iteration1} and~\eqref{def1th} that 
for all $m$ and $n\ge m$
\begin{align*}
\Theta_{n}-\frac 1 2 
&=\frac{\tau_{n-1}}{\tau_{n}} \Theta_{n-1}+\frac{\sigma_{n}}{\tau_{n}}\psi(\Theta_{n-1})-\frac 1 2+\e_{n}\frac{\sqrt{\sigma_{n}P_{n-1}(1-P_{n-1})}}{\tau_{n}} \\
&=\kappa_{n-1}(\Theta_{n-1})\Big(\Theta_{n-1}-\frac 1 2\Big)+\e_{n}\frac{\sqrt{\sigma_{n}P_{n-1}(1-P_{n-1})}}{\tau_{n}}.
\end{align*}
Iterating this procedure we get 
\begin{align}
\Theta_{n}-\frac 1 2
&=\Big[\prod_{j=m}^{n-1}\kappa_j(\Theta_{j})\Big] \Big(\Theta_{m}-\frac 1 2\Big)
+\sum_{k=m+1}^{n}\Big[\prod_{j=k}^{n-1}\kappa_j(\Theta_{j})\Big] \e_{k}\frac{\sqrt{\sigma_{k}P_{k-1}(1-P_{k-1})}}{\tau_{k}}\notag\\
&=\Big[\prod_{j=m}^{n-1}\kappa_j(\Theta_{j})\Big] \Big(\Theta_{m}-\frac 1 2
+\sum_{k=m+1}^n\Big[\prod_{j=m}^{k-1}\frac{1}{\kappa_j(\Theta_{j})}\Big] \e_{k}\frac{\sqrt{\sigma_{k}P_{k-1}(1-P_{k-1})}}{\tau_{k}}\Big).
\label{ka4p}
\end{align}

Observe that on the event $\mathcal{H}_{m}$
\begin{align}
\label{bobo}
\kappa_j(\Theta_j)=\frac{\tau_j}{\tau_{j+1}}+\frac{\sigma_{j+1}}{\tau_{j+1}}\Big[\psi'(1/2)+O(\delta_j)\Big]
=\frac{\tau_j}{\tau_{j+1}}+\frac{\sigma_{j+1}}{\tau_{j+1}}\alpha+\frac{\sigma_{j+1}}{\tau_{j+1}}O(\delta_j)
\end{align}
as $j\to\infty$ uniformly on the probability space. 
Denote 
\begin{align}
\label{rhorho}
\pi_{m,k}=\prod_{j=m}^{k-1}\frac{\tau_j+\alpha\sigma_{j+1}}{\tau_{j+1}}=\exp\Big\{\sum_{j=m+1}^{k}\log\Big(1+(\alpha-1)\frac{\sigma_j}{\tau_j}\Big)\Big\}.
\end{align}
It follows from~\eqref{bobo} that, as $m\to\infty$, uniformly in $\omega$ and $k\ge m$
\begin{align}
\prod_{j=m}^{k-1}\kappa_j(\Theta_{j})
=\pi_{m,k} \exp\Big\{\sum_{j=m}^{k-1}\log\Big(1+\frac{\sigma_{j+1}}{\tau_{j}+\alpha\sigma_{j+1}}O(\delta_j)\Big)\Big\}
&=\pi_{m,k} \big(1+o(1)\big),
\label{ka2p}
\end{align}
since by~\eqref{ser1} we have 
\begin{align*}
\sum_{j=m}^{k-1}\frac{\delta_j\sigma_{j+1}}{\tau_{j}+\alpha\sigma_{j+1}}
\le \sum_{j=m}^{\infty}\delta_j\, \frac{\sigma_{j+1}}{\tau_{j+1}}=o(1).
\end{align*}


Further, on $\mathcal{H}_{m}$ we have, as $m\to\infty$, uniformly in $\omega$ and $k\ge m$ 
\begin{align}
\label{ka3p}
\sqrt{P_{k-1}(1-P_{k-1})}=\sqrt{\psi(\Theta_{k-1})(1-\psi(\Theta_{k-1}))}=1/2+O(\delta_{k-1})=1/2+o(1).
\end{align}
Substituting~\eqref{ka2p} and~\eqref{ka3p} into~\eqref{ka4p} we obtain, as $m\to\infty$,  uniformly in $\omega$
and $n\ge m$
\begin{align*}
&\Theta_{n}-\frac 1 2
=(1+o(1))\pi_{m,n} \Big(\Theta_{m}-\frac 1 2
+\frac{1+o(1)}{2}\sum_{k=m+1}^n \e_{k}\frac{\sqrt{\sigma_{k}}}{\pi_{m,k}\tau_{k}}\Big).
\end{align*}

Let $(B^{\ssup 0}_{\ell})_{\ell\in\N}$ be a sequence of independent Binomial random variables with parameter $\frac 1 2$ and size~$\sigma_{\ell}$, and let, for all $\ell\in\N$, 
\begin{align*}
\e_{\ell}^{\ssup 0}=\frac{B^{\ssup 0}_{\ell}-\sigma_{\ell}/2}{\sqrt{\sigma_{\ell}/4}}.
\end{align*}
Denote 
\begin{align*}
\hat{\e}_{\ell}=\e_{\ell}\one\big\{|\Theta_{{\ell}-1}-1/2|<1/4\big\}+\e_{\ell}^{\ssup 0}\one\big\{|\Theta_{{\ell}-1}-1/2|\ge 1/4\big\}.
\end{align*}
On the event $\mathcal{H}_{m}$ we have, as $m\to\infty$,  uniformly in $\omega$
and $n\ge m$
\begin{align}
\Theta_{n}-\frac 1 2
&=(1+o(1))\pi_{m,n} 
 \Big(\Theta_{m}-\frac 1 2
+\frac{1+o(1)}{2}\sum_{k=m+1}^n\hat{\e}_{k}\frac{\sqrt{\sigma_{k}}}{\pi_{m,k}\tau_{k}} \Big).
\label{ka7p}
\end{align}
\smallskip

It is easy to see that for a Binomial random variable $\text{Bin}(p,n)$ we have 
\begin{align*}
\E \exp\Big\{it\frac{\text{\rm Bin}(p,n)-np}{\sqrt{np(1-p)}}\Big\}
&=\exp\Big\{n\log\Big(1-p+p\exp\Big\{\frac{it}{\sqrt{np(1-p)}}\Big\}\Big)-\frac{it\sqrt{np}}{\sqrt{1-p}}\Big\}\\
&=\exp\Big\{n\log\Big(1+\frac{it\sqrt{p}}{\sqrt{n(1-p)}}-
\frac{t^2 }{2n(1-p)}+O\big(\frac{t^3}{n\sqrt n}\big)\Big)-\frac{it\sqrt{np}}{\sqrt{1-p}}\Big\}\\
&=\exp\Big\{-
\frac{t^2 }{2}+nO\big(\frac{t^3}{n\sqrt n}\big)\Big\}
\end{align*}
as $n\to\infty$, uniformly in ${\frac 1 4<p<\frac 3 4}$ and $t$. Hence 
\begin{align}
\label{ka9p}
\E_{\mathcal{F}_{\ell-1}}e^{it\hat{\e}_{\ell}}&=\exp\Big\{-
\frac{t^2 }{2}+\sigma_{\ell}O\big(\frac{t^3}{\sigma_{\ell}\sqrt{\sigma_{\ell}}}\big)\Big\}
\end{align}
as $\ell\to\infty$, uniformly in $\omega$ and $t$. 
Denote
\begin{align}
\label{mumu}
\mu_{m,n}=\Big[\sum_{k=m+1}^n\frac{\sigma_{k}}{\pi^2_{m,k}\tau^2_{k}}\Big]^{1/2}.
\end{align}
Using~\eqref{ka9p} we have for each fixed $t$, as $m\to\infty$, uniformly in $\omega$ and $n\ge m$,
\begin{align}
\label{ka6p}
\E_{\mathcal{F}_m}\exp\Big\{\frac{it}{\mu_{m,n}}\sum_{k=m+1}^n\hat{\e}_{k}\frac{\sqrt{\sigma_{k}}}{\pi_{m,k}\tau_{k}}\Big\}
&=\exp\Big\{-\frac{t^2}{2}+O(1)\frac{1}{\mu_{m,n}^3}\sum_{k=m+1}^n\frac{\sigma_{k}}{\pi^3_{m,k}\tau^3_{k}}\Big\}.
\end{align}
Let us show that the term next to $O(1)$ tends to zero as $m\to\infty$. We need to consider two
cases:  when  $(\sigma_i)$ tends to infinity and when it is bounded. 
\smallskip

First, assume that $\sigma_i\to \infty$. Observe that since $(x_1+\cdots+x_m)^{\frac 3 2}\ge x_1^{\frac 3 2}+\cdots+x_m^{\frac 3 2}$ for all $m$ and all non-negative $x_1,\dots,x_m$, we have 
\begin{align*}
\mu^3_{m,n}\ge \sum_{k=m+1}^n\frac{\sigma^{3/2}_{k}}{\pi^3_{m,k}\tau^3_{k}}
\ge \min_{m<k\le n}\sqrt{\sigma_{k}}\,\sum_{k=m+1}^n\frac{\sigma_{k}}{\pi^3_{m,k}\tau^3_{k}}
\end{align*}
and hence 
\begin{align*}
\frac{1}{\mu_{m,n}^3}\sum_{k=m+1}^n\frac{\sigma_{k}}{\pi^3_{m,k}\tau^3_{k}}\le \max_{m<k\le n}\frac{1}{\sqrt{\sigma_{k}}}\to 0
\end{align*}
as $m\to\infty$ uniformly in $n$.  
\smallskip

Second, assume that $\sigma_i$ is bounded by a constant $\sigma$. Hence $i\le \tau_i\le \sigma (i+1)\le 2\sigma i$ for all $i$. We have 
\begin{align}
\pi_{m,k}
&\le \exp\Big\{\sum_{j=m+1}^{k}\log\Big(1+(\alpha-1)\frac{\sigma}{j}\Big)\Big\}
= \exp\Big\{\big[(\alpha-1)\sigma+o(1)\big]\sum_{j=m+1}^k\frac 1 j\Big\}\notag\\
&= \exp\Big\{\big[(\alpha-1)\sigma+o(1)\big]\big(\log k-\log m+o(1)\big)\Big\}=\frac{k^{(\alpha-1)\sigma+o(1)}}{m^{(\alpha-1)\sigma+o(1)}}.
\label{upbo}
\end{align}
and, similarly, 
\begin{align}
\label{lobo}
\pi_{m,k}
&\ge\frac{k^{\frac{\alpha-1}{\sigma}+o(1)}}{m^{\frac{\alpha-1}{\sigma}+o(1)}}.
\end{align}
as $m\to\infty$ uniformly in $k\ge m$. 
Observe that for all $\gamma>1$
\begin{align*}
\sum_{k=m+1}^n\frac{1}{k^{\gamma+o(1)}}=\frac{1}{m^{\gamma-1+o(1)}}
\end{align*}
as $m\to\infty$ uniformly in $n\ge m^2$. 
Using this we obtain by~\eqref{upbo}
\begin{align}
\label{mumum}
\mu_{m,n}^3\ge\frac{m^{3(\alpha-1)\sigma+o(1)}}{8\sigma^3}\Big[\sum_{k=m+1}^n 
\frac{1}{k^{2(\alpha-1)\sigma+2+o(1)}}\Big]^{3/2}
=m^{3(\alpha-1)\sigma-\frac 3 2 (2(\alpha-1)\sigma+1)+o(1)}=m^{-\frac 3 2 +o(1)}
\end{align}
and by~\eqref{lobo}
\begin{align}
\label{mumumu}
\sum_{k=m+1}^n\frac{\sigma_{k}}{\pi^3_{m,k}\tau^3_{k}}
&\le \sigma m^{\frac{3(\alpha-1)}{\sigma}+o(1)}\sum_{k=m+1}^n \frac{1}{k^{\frac{3(\alpha-1)}{\sigma}+3+o(1)}}
=m^{-2+o(1)}
\end{align}
as $m\to\infty$ uniformly in $n\ge m^2$. 
Combining~\eqref{mumum} and~\eqref{mumumu} we obtain 
\begin{align*}
\frac{1}{\mu_{m,n}^3}\sum_{k=m+1}^n\frac{\sigma_{k}}{\pi^3_{m,k}\tau^3_{k}}=m^{-1/2+o(1)}\to 0
\end{align*}
as $m\to\infty$ uniformly in $n\ge m^2$.  
\smallskip

Let $(n_m)$ be an $\N$-valued sequence satisfying $n_m\ge m^2$. We will need $(n_m)$ to grow sufficiently fast but will specify 
this condition later.
It follows from~\eqref{ka7p} and~\eqref{ka6p} that on the event $\mathcal{H}_m$
\begin{align*}
\Theta_{n_m}-\frac 1 2
&=(1+o(1))\pi_{m,n_m} \Big(\Theta_{m}-\frac 1 2
+\frac{1+o(1)}{2}\,\mu_{m,n_m} N_{m,n_m}\Big),
\end{align*}
where 
$N_{m,n_m}$ is a random variable, which conditionally on $\Theta_m$ converges weakly to a standard normal random variable $N$. Hence 
\begin{align*}
\P(\mathcal{H}_m)\le \P\Big((1+o(1))\pi_{m,n_m} \Big(\Theta_{m}-\frac 1 2
+\frac{1+o(1)}{2}\,\mu_{m,n_m} N_{m,n_m}\Big)\in [-\delta_{n_m},\delta_{n_m}]\Big)\to 0
\end{align*}
as required if 
\begin{align}
\label{fin}
\frac{\pi_{m,n_m}\mu_{m,n_m}}{\delta_{n_m}}\to \infty.
\end{align}

Finally, to prove~\eqref{fin}, we estimate the sum in~\eqref{mumu} by the first term and obtain using $\sigma_{m+1}\ge 1$, $\pi_{m.m+1}\le \alpha$, and~\eqref{rhorho}
\begin{align*}
\pi_{m,n_m}\mu_{m,n_m}
\ge \pi_{m,n_m}\,\frac{\sqrt{\sigma_{m+1}}}{\pi_{m,m+1}\tau_{m+1}}
\ge\frac{1}{\alpha\tau_{m+1}}\exp\Big\{\sum_{j=m+1}^{n_m}\log\Big(1+(\alpha-1)\frac{\sigma_j}{\tau_j}\Big)\Big\}.
\end{align*}
The series 
\begin{align*}
\sum_{j=0}^{\infty} \log\Big(1+(\alpha-1)\frac{\sigma_j}{\tau_j}\Big)=\infty
\end{align*}
diverges by Lemma~\ref{1101series} and hence we can choose $(n_m)$ to grow sufficiently fast to guarantee  
\begin{align*}
\pi_{m,n_m}\mu_{m,n_m}\to \infty,
\end{align*}
which implies~\eqref{fin}.
\end{proof}

\bigskip



\section{Dominance}

\label{s:dom}

The aim of this section is to prove almost sure dominance in the positive  feedback scenario. We will rely on Proposition~\ref{devp} for showing that the proportion $\Theta_n$ of the balls in the first bin does not get stuck at the equilibrium. To do so we need to pick the sequence of deviations $(\delta_n)$ satisfying the assumption~\eqref{ser1}. Let 
\begin{align}
\label{d:deltan1}
\delta_n=\frac{1}{\log^2\tau_n},\quad n\in\N_0.
\end{align}

\begin{lemma} 
\label{1101m}
Suppose {\rm (R)} is satisfied. Then
\begin{align}
\label{tmp1}
\sum_{n=0}^{\infty}\delta_n\frac{\sigma_{n+1}}{\tau_{n+1}}<\infty.
\end{align}
\end{lemma}

\begin{proof} First, suppose $(\rho_n)$ is bounded. We have 
\begin{align*}
\log \tau_{n+1}=\log\tau_{n}+\log (1+\rho_{n})\sim \log\tau_{n}
\end{align*}
as $n\to\infty$, and the convergence of the above series follows from 
\begin{align*}
\sum_{n=0}^{\infty}\frac{\sigma_{n+1}}{\tau_{n+1}\log^2\tau_{n+1}}\le \int_{\tau_0}^{\infty}\frac{dx}{x\log^2x}<\infty.
\end{align*}

Second, suppose $\rho_n\to\infty$. Then there exists $m\in \N$ such that for all $n\ge m$ we have $\rho_n\ge 2$
and hence $\tau_{n}\ge \sigma_{n}\ge 2\tau_{n-1}\ge\cdots\ge 2^{n-m}\tau_m$. Together with $\sigma_{n+1}\le \tau_{n+1}$ this implies 
\begin{align*}
\delta_n\frac{\sigma_{n+1}}{\tau_{n+1}}\le \frac{1}{\log^2 \tau_{n}}\le \frac{1}{((n-m)\log 2+\log\tau_m)^2}
\sim \frac{1}{n^2\log^22}
\end{align*}
as $n\to\infty$ implying convergence of the series~\eqref{tmp1}. 
\end{proof}
\smallskip

Now we are ready to prove dominance. Our strategy will be as follows. First we pick a time when $\Theta_n$ deviates from the equilibrium. 
Then we observe that if that time is large enough then the 
future fluctuations of the martingale part of $\Theta_n$ will be small with high probability, and will keep $\Theta_n$ away from the equilibrium. At the same time, the bias caused by the positive feedback will move $\Theta_n$ away from the equilibrium, and its power will be sufficient to bring it to zero or one, respectively.   

\begin{proof}[Proof of Theorem~\ref{main0}]
%

\smallskip

Fix $r\ge e^4$ and let 
\begin{align*}
\eta=\inf\big\{n\ge r: |\Theta_n-1/2|>\delta_n\big\}
\end{align*}
be the first time $\Theta_n$ significantly deviates from the equilibrium after time $r$. Observe that $\eta$
is finite almost surely by Proposition~\ref{devp} and Lemma~\ref{1101m}. Due to symmetry it suffices to consider the event 
\begin{align}
\label{lo3}
\mathcal{E}=\big\{\Theta_{\eta}<1/2-\delta_{\eta}\big\}.
\end{align}
and show that on $\Theta_n\to 0$ on $\mathcal{E}$.
\smallskip

For each $n\in\N_0$ we have by~\eqref{def1th}, \eqref{def1t}, and~\eqref{defp}
\begin{align*}
\Theta_{n+1}=\frac{\tau_n}{\tau_{n+1}}\Theta_n+\frac{1}{\tau_{n+1}}B_{n+1}=\Theta_n+\frac{B_{n+1}-\sigma_{n+1}P_n}{\tau_{n+1}}
-\frac{\sigma_{n+1}}{\tau_{n+1}}\big(\Theta_n-\psi(\Theta_n)\big).
\end{align*}
Hence for each $k\ge \eta$ we have
\begin{align}
\label{lo2}
\Theta_{\eta+n}
=\Theta_{\eta}+M_n-R_n,
\end{align}
where, for each $n\in\N_0$, 
\begin{align*}
M_n=\sum_{k=\eta+1}^{\eta+n}\frac{B_{k}-\sigma_{k}P_{k-1}}{\tau_k}
\qquad\text{and}\qquad
R_n=\sum_{k=\eta+1}^{\eta+n}\frac{\sigma_{k}}{\tau_{k}}\big(\Theta_{k-1}-\psi(\Theta_{k-1})\big).
\end{align*}
It is easy to see that $(M_n)$ is a martingale with respect to the filtration $(\mathcal{F}_{\eta+n})$. Moreover, it is bounded in $L^2$ as for all $n$
\begin{align}
\label{bb1}
\E_{\mathcal{F}(\eta)} M_n^2
&=\E_{\mathcal{F}(\eta)} M_{n-1}^2+
\E_{\mathcal{F}(\eta)}\Big[\frac{\sigma_{\eta+n}}{\tau_{\eta+n}^2}
P_{\eta+n-1}(1-P_{\eta+n-1})\Big]\notag\\
&\le \E_{\mathcal{F}(\eta)} M_{n-1}^2+\frac{\sigma_{\eta+n}}{\tau_{\eta+n}^2}
\le\cdots\le \sum_{k=\eta+1}^{\eta+n}\frac{\sigma_{k}}{\tau^2_{k}}\le\int_{\tau_{\eta}}^{\infty}\frac{dx}{x^2}=1/\tau_{\eta}.
\end{align} 
Hence $(M_n)$ converges almost surely conditionally on $\mathcal{F}_{\eta}$. 
Denote 
\begin{align*}
\mathcal{S}=\Big\{\sup_{n\in\N}M_n\le  \frac{\delta_{\eta}}{2}\Big\}.
\end{align*}
By Doob's submartingale inequality we have using~\eqref{bb1}
\begin{align*}
\P_{\mathcal{F}(\eta)}\Big(\max_{1\le k\le n}M_k> \frac{\delta_{\eta}}{2}\Big)
&\le \P_{\mathcal{F}(\eta)}\Big(\max_{1\le k\le n} M^2_k> \frac{\delta_{\eta}^2}{4}\Big)
\le \frac{4}{\delta_{\eta}^2} \E_{\mathcal{F}(\eta)} M_n^2
\le \frac{4\log^4\tau_{\eta}}{\tau_{\eta}}
\le\frac{4\log^4\tau_{r}}{\tau_{r}},
\end{align*}
as the function $x\mapsto \frac{\log^4x}{x}$ is decreasing on $[e^4,\infty)$ and $r\ge e^4$. This implies
\begin{align}
\label{lo5}
\P(\mathcal{S}^c)
\le\frac{4\log^4\tau_{r}}{\tau_{r}}.
\end{align}

Let us prove by induction that 
\begin{align}
\label{bb2}
\Theta_{\eta+n}<\frac 1 2-\frac{\delta_{\eta}}{2} 
\end{align}
for all $n\in\N_0$ on the event $\mathcal{S}\cap \mathcal{E}$. 
Indeed, for $n=0$ it follows from~\eqref{lo3}. Suppose it is true for all indices between $0$ and $n-1$. Since $x\ge \psi(x)$ for all $x\in [0,1/2]$ we have $\Theta_{k-1}\ge \psi(\Theta_{k-1})$
for all $\eta+1\le k\le \eta+n$ and hence $R_n\ge 0$. By~\eqref{lo2} this implies 
\begin{align*}
\Theta_{\eta+n}\le \Theta_{\eta}+M_n<\frac 1 2-\delta_{\eta}+\frac{\delta_{\eta}}{2}=\frac 1 2-\frac{\delta_{\eta}}{2}
\end{align*}
as required. 
\smallskip

Observe that~\eqref{bb2} and $x\ge \psi(x)$ for all $x\in [0,1/2]$ imply that $(R_n)$ is increasing on $\mathcal{S}\cap \mathcal{E}$.  Since $(M_n)$ converges we obtain by~\eqref{lo2} that $(\Theta_n)$
converges on $\mathcal{S}\cap \mathcal{E}$. Let us show that on $\mathcal{S}\cap \mathcal{E}$
\begin{align}
\label{lili}
\lim_{n\to\infty} \Theta_{n}=0.
\end{align}
Indeed, if the limit $\Theta$ is positive for some $\omega\in\mathcal{S}\cap \mathcal{E}$ then $\Theta>\psi(\Theta)$ and 
\begin{align*}
\frac{\sigma_{k}}{\tau_{k}}\big[\Theta_{k-1}-\psi(\Theta_{k-1})\big]\sim \frac{\sigma_{k}}{\tau_{k}}\big[\Theta-\psi(\Theta)\big], 
\end{align*}
implying 
\begin{align*}
\sum_{k=\eta+1}^{\infty}\frac{\sigma_{k}}{\tau_{k}}\big[\Theta_{k-1}-\psi(\Theta_{k-1})\big]=\infty
\end{align*}
by Lemma~\ref{1101series}. This means that 
$R_n\to\infty$, which by~\eqref{lo2} implies $\Theta= -\infty$, which is clearly impossible.
\smallskip
 
Since $r$ is arbitrary, it follows from~\eqref{lo5} that~\eqref{lili} holds on $\mathcal{E}$ almost surely. 
\end{proof}
\bigskip




\section{Some properties of $(\sigma_n)$ and $(\e_n)$}

\label{s:pro}


In this section we collect some elementary results about 
the sequence of sample sizes $(\sigma_n)$ and the sequence of random 
noises $(\e_n)$ that will be later used for the subcritical regime.
\smallskip

For each $n\in\N$, denote 
\begin{align}
\label{d:lambdan}
\lambda_n=\frac{\sigma_{n+1}\sigma_{n-1}^{\alpha}}{\sigma_n^{\alpha+1}}.
\end{align}
Observe that if $\rho_n\to\infty$ we have 
\begin{align}
\label{st}
\lim_{n\to\infty}\frac{\tau_n}{\sigma_n}=\lim_{n\to\infty}\Big(\frac{\tau_{n-1}}{\sigma_n}+1\Big)=1. 
\end{align}


The following lemma explains the nature of $\lambda$
and shows that for regular enough sequences it is always equal to zero
in the subcritical regime.

\begin{lemma}
\label{l:regular}
If $\rho_n\to\infty$, $\theta=0$, and 
$\lambda_n\to \lambda$ then $\lambda=0$.  
\end{lemma}

\begin{proof} Suppose $\lambda\in (0,\infty]$. 
It follows from~\eqref{st} that 
\begin{align}
\label{r1}
\lim_{n\to\infty}\frac{\rho_n}{\rho_{n-1}^{\alpha}}
=\lim_{n\to\infty}\frac{\sigma_{n+1}\tau_{n-1}^{\alpha}}{\tau_n\sigma_n^{\alpha}}
=\lim_{n\to\infty}\lambda_n=\lambda.
\end{align}

Let $\e\in (0,\lambda)$
and let $k\in \N$ be such that $\e^{\frac{1}{\alpha-1}}\rho_k>1$ and, for all $n> k$, 
$$\rho_n>\e\rho_{n-1}^{\alpha},$$
which is possible by~\eqref{r1}. Iterating, we have for all $n>k$
\begin{align*}
\rho_n>\e^{1+\alpha+\cdots+\alpha^{n-k-1}}\rho_k^{\alpha^{n-k}}=\e^{\frac{\alpha^{n-k}-1}{\alpha-1}}\rho_k^{\alpha^{n-k}}.
\end{align*}
Using $\tau_{n+1}\ge\sigma_{n+1}=\tau_n\rho_n\ge \rho_n$ 
we obtain 
\begin{align*}
0=\alpha\theta=\lim_{n\to\infty}\alpha^{-n}\log\tau_{n+1}
\ge \lim_{n\to\infty}\alpha^{-n}\log\rho_{n}
\ge \alpha^{-k}\log\big(\e^{\frac{1}{\alpha-1}}\rho_k\big)>0
\end{align*}
leading to a contradiction. 
\end{proof}

Recall from Lemma~\ref{diverges} that the series~\eqref{div7} plays an important r\^ ole for non-occurrence of monopoly. In the following lemma we explore the behaviour of that series in various cases of the subcritical regime. 

\begin{lemma} 
\label{omegazero}
If $(\rho_n)$ is bounded then $\theta=0$ and there is $c>0$ such that for all $m\in\N_0$
\begin{align*}
\sum_{n=m}^{\infty}\frac{\sigma_{n+1}}{\tau_n^{\alpha}}\le \frac{c}{\tau_m^{\alpha-1}}.
\end{align*} 
If $\rho_n\to\infty$, $\theta=0$, and $\lambda<1$ then
\begin{align}
\label{sercon}
\sum_{n=0}^{\infty}\frac{\sigma_{n+1}}{\tau_n^{\alpha}}<\infty.
\end{align} 
If $\rho_n\to\infty$, $\theta=0$, and $\lambda>1$ then
\begin{align}
\label{sercon1}
\sum_{n=0}^{\infty}\frac{\sigma_{n+1}}{\tau_n^{\alpha}}=\infty.
\end{align} 
\end{lemma}

\begin{proof} 
First, suppose $(\rho_n)$ is bounded by some constant $\rho$. We have 
\begin{align*}
\alpha^{-n}\log\tau_n=\alpha^{-n}\Big(\log\tau_0+\sum_{i=0}^{n-1}\log (1+\rho_i)\Big)
\le \alpha^{-n}\big(\log\tau_0+n\log(1+\rho)\big)\to 0
\end{align*}
implying $\theta=0$. Further, we have 
\begin{align*}
\sum_{n=m}^{\infty}\frac{\sigma_{n+1}}{\tau_n^{\alpha}}
=\sum_{n=m}^{\infty}\frac{(1+\rho_n)^{\alpha}\sigma_{n+1}}{\tau_{n+1}^{\alpha}}
\le (1+\rho)^{\alpha}\sum_{n=m}^{\infty}\frac{\sigma_{n+1}}{\tau_{n+1}^{\alpha}}
\le (1+\rho)^{\alpha}\int_{\tau_m}^{\infty}\frac{dx}{x^{\alpha}}
=\frac{(1+\rho)^{\alpha}}{\alpha-1}\cdot\frac{1}{\tau_m^{\alpha-1}}.
\end{align*}
\smallskip

Now suppose $\rho_n\to\infty$ and $\theta=0$. 
Observe that 
\begin{align}
\label{ooo}
\frac{\sigma_{n+1}}{\tau_n^{\alpha}}
=\frac{\sigma_n}{\tau_{n-1}^{\alpha}}\cdot\frac{\rho_n}{\rho_{n-1}^{\alpha}} \Big(\frac{\rho_{n-1}}{1+\rho_{n-1}}\Big)^{\alpha-1}.
\end{align}
Using $\rho_n\to\infty$ and~\eqref{st} we obtain similarly to~\eqref{r1} 
\begin{align*}
\limsup_{n\to\infty}\frac{\rho_n}{\rho_{n-1}^{\alpha}} \Big(\frac{\rho_{n-1}}{1+\rho_{n-1}}\Big)^{\alpha-1}=\lambda,
\end{align*}
which implies convergence~\eqref{sercon} and divergence~\eqref{sercon1} by the ratio test together with~\eqref{ooo}. 
\end{proof}
\smallskip

The next lemma gives an elementary upper bound for the noise terms $(\e_n)$.

\begin{lemma} 
\label{le1}
Almost surely, $\e_{n+1}\le n$ eventually for all $n$. 
\end{lemma}

\begin{proof} By L\' evy's extension of the Borel-Cantelli Lemmas, see~\cite[\S 12.15]{DW},
\begin{align*}
\{\e_{n+1}> n\text{ i.o.}\}=\Big\{\sum_{n=1}^{\infty}\P_{\mathcal{F}_n}(\e_{n+1}> n)=\infty\Big\}. 
\end{align*}
By Chebychev's inequality and using~\eqref{norm} we have
\begin{align*}
\P_{\mathcal{F}_n}(\e_{n+1}> n)\le\P_{\mathcal{F}_n}(\e_{n+1}^2> n^2)\le \frac{1}{n^2}
\end{align*}
implying that the series converges almost surely. 
\end{proof}

The final result of this section is a CLT statement about the random series of $(\e_n)$. 

\begin{lemma} 
\label{l:liminf0}
Let $(\xi_n)$ be a sequence of random variables adapted to the filtration $(\mathcal{F}_n)$. Suppose 
\begin{align}
\label{bd}
|\xi_i|\le \zeta_n a_i\qquad\text{for all }i\ge n,
\end{align}
for all $n$ almost surely, where $(\zeta_n)$ is an almost surely positive and square-integrable $(\mathcal{F}_n)$-adapted sequence, and $(a_n)$ is a deterministic square-summable sequence. Then, almost surely,  
\begin{align}
\label{se10}
\sum_{i=0}^{\infty}\xi_i\e_{i+1}<\infty
\end{align}
and
\begin{align}
\label{se20}
\liminf_{n\to\infty}\
\frac{1}{\zeta_n\sqrt{A_n}}\sum_{i=n}^{\infty}\xi_i\e_{i+1}<\infty,
\end{align}
where 
\begin{align*}
A_n=\sum_{i=n}^{\infty}a_i^2.
\end{align*}
\end{lemma}

\begin{proof} 
Let $n\in\N_0$. For each $m\ge n$, denote 
\begin{align*}
S^{\ssup n}_m=\sum_{i=n}^m\xi_i\e_{i+1}.
\end{align*}
Observe that $S^{\ssup n}$ is a martingale bounded in $L^2$ since
by~\eqref{se10} for all $m$
\begin{align}
\label{mo20}
\E \big(S_m^{\ssup n}\big)^2
\le \E\zeta_n^2\, \sum_{i=n}^{\infty}a_i^2,
\end{align}
which is finite as $(a_i)$ is summable and $\zeta_n$ is square-integrable.  
Hence, as $m\to\infty$, $S^{\ssup n}_m$ converges almost surely. 
This in particular implies~\eqref{se10}. 
\smallskip

Observe that~\eqref{se20} is equivalent to showing that 
\begin{align}
\label{se205}
\P\Big(\bigcap_{k=1}^{\infty}\bigcup_{N=1}^{\infty}\bigcap_{n\ge N} \Big\{
\frac{1}{\zeta_n\sqrt{A_n}}\sum_{i=n}^{\infty}\xi_i\e_{i+1}>k
\Big\}\Big)=0,
\end{align}
or, equivalently,
\begin{align}
\label{mo30}
\lim_{k\to\infty}\lim_{N\to\infty}\P\Big(\bigcap_{n\ge N} \Big\{
\frac{1}{\zeta_n\sqrt{A_n}}\sum_{i=n}^{\infty}\xi_i\e_{i+1}>k
\Big\}\Big)=0.
\end{align}
By Chebychev's inequality and using~\eqref{norm}, \eqref{bd} and $L^2$-boundedness of the martingales $S^{\ssup N}$ conditionally on $\mathcal{F}_{N}$ in the same way it was done in~\eqref{mo20}, we have
\begin{align*}
\P\Big(\bigcap_{n\ge N} &\Big\{\frac{1}{\zeta_n\sqrt{A_n}}\sum_{i=n}^{\infty}\xi_i\e_{i+1}>k\Big\}\Big)
\le \P\Big(\frac{1}{\zeta_N\sqrt{A_N}}\sum_{i=N}^{\infty}\xi_i\e_{i+1}>k\Big)\\
&\le \E\,\P_{\mathcal{F}_{N}}\Big(\Big(\sum_{i=N}^{\infty}\xi_i\e_{i+1}\Big)^2> k^2\zeta_N^2A_N\Big)
\le \frac{1}{k^2}\E\Big[\frac{1}{\zeta_N^2A_N}\E_{\mathcal{F}_{N}}\Big(\sum_{i=N}^{\infty}\xi_i\e_{i+1}\Big)^2\Big]
\le \frac{1}{k^2}.
\end{align*}
Combining this with~\eqref{mo30} we obtain~\eqref{se205}.
\end{proof}
\bigskip


\section{Subcritical regime}

\label{s:sub}

The aim of this section is to prove Theorem~\ref{main1}. We will split it in two propositions corresponding to 
the cases when $(\rho_n)$ is bounded and when it tends to infinity with $\lambda<1$. The remaining case of 
$\rho_n\to \infty$ with $\lambda>1$ will follow easily from the same method as in Section~\ref{s:nomo}.
\smallskip

It follows from~\eqref{psibounds} and~\eqref{iteration1} that for all $n\in\N_0$
\begin{align}
T_{n+1}
&\le T_n+2^{\alpha-1}\sigma_{n+1}\Theta_n^{\alpha} 
+\e_{n+1}\sqrt{\sigma_{n+1}P_n(1-P_n)}\notag\\
&= T_n+2^{\alpha-1}T_n^{\alpha}\frac{\sigma_{n+1}}{\tau_n^{\alpha}} 
+\e_{n+1}\sqrt{\sigma_{n+1}P_n(1-P_n)}.
\label{itit}
\end{align}
This iterative upper bound will play an important role in showing monopoly. 
\smallskip

The first proposition deals with the case when $(\rho_n)$ is bounded. This is an easier case as $(\sigma_n)$ is growing not too fast so that the sums 
\begin{align}
\label{aprox}
\sum_{i=n}^{\infty}\frac{\sigma_{i+1}}{\tau_i^{\alpha}}\approx \int_{\tau_n}^{\infty}\frac{dx}{x^{\alpha}}=\frac{1}{\alpha-1}\cdot \frac{1}{\tau_n^{\alpha-1}}
\end{align}
can be approximated accurately enough by the corresponding integrals. 

\begin{prop} 
\label{p:mod1}
Suppose $\alpha>1$, {\rm (S)} is satisfied, and 
$(\rho_n)$ is bounded. Then 
$\P(\mathcal{M})=1$.
\end{prop}

\begin{proof}
%
By symmetry and Theorem~\ref{main0}, without loss of generality it suffices to consider the event $\mathcal{E}=\{\Theta_n\to 0\}$ and prove that 
$T_n$ is bounded on $\mathcal{E}$.
\smallskip

Consider the event $\mathcal{E}$ and assume that $T_n\to\infty$.
Then for all $n\in\N_0$
\begin{align}
\label{ko1}
\sum_{i=n}^{\infty}\frac{T_{i+1}-T_i}{T_i^{\alpha}}\ge \int_{T_n}^{\infty}\frac{dx}{x^{\alpha}}=\frac{1}{\alpha-1}\cdot\frac{1}{T_n^{\alpha-1}}.
\end{align}


On the other hand, it follows from~\eqref{itit} that 
\begin{align}
\label{fdif1}
\frac{T_{i+1}-T_i}{T_i^{\alpha}}\le 2^{\alpha-1}\frac{\sigma_{i+1}}{\tau_i^{\alpha}}+\xi_{i}\e_{i+1},
\end{align}
where 
\begin{align*}
\xi_{i}=T_i^{-\alpha}\sqrt{\sigma_{i+1}P_i(1-P_i)}.
\end{align*}

By~\eqref{fdif1} and Lemma~\ref{omegazero} we have with some $c>0$
\begin{align}
\label{ko2}
\sum_{i=n}^{\infty}\frac{T_{i+1}-T_i}{T_i^{\alpha}}
\le 2^{\alpha-1}\sum_{i=n}^{\infty}\frac{\sigma_{i+1}}{\tau_i^{\alpha}}
+\sum_{i=n}^{\infty}\xi_{i}\e_{i+1}
\le \frac{c}{\tau_n^{\alpha-1}}+\sum_{i=n}^{\infty}\xi_{i}\e_{i+1}.
\end{align}

Combining this with~\eqref{ko1} we obtain   
\begin{align}
\label{ko5}
\frac{1}{\alpha-1}\cdot\frac{1}{T_n^{\alpha-1}}\le\frac{c}{\tau_n^{\alpha-1}}+
\sum_{i=n}^{\infty}\xi_{i}\e_{i+1}.
\end{align}
Observe that by~\eqref{psibounds} and using $T_i\ge T_n$ on the whole probability space we have 
\begin{align*}
|\xi_{i}|\le T_i^{-\alpha}\sqrt{\sigma_{i+1}2^{\alpha-1}\Theta_i^{\alpha}}
=T_i^{-\frac{\alpha}{2}}\sqrt{2^{\alpha-1}\frac{\sigma_{i+1}}{\tau_i^{\alpha}}}
\le \zeta_n a_{i}
\end{align*}
for all $n$ and all $i\ge n$, where 
\begin{align*}
\zeta_n=T_n^{-\frac{\alpha}{2}}
\qquad\text{and}\qquad
a_i=\sqrt{2^{\alpha-1}\frac{\sigma_{i+1}}{\tau_i^{\alpha}}}.
\end{align*}
The sequence $(\zeta_n)$ is clearly $(\mathcal{F}_n)$-adapted and square-integrable since it is bounded by one. 
By Lemma~\ref{omegazero} the sequence $(a_i)$ is square-summable, and 
\begin{align*}
A_{n}=\sum_{i=n}^{\infty}a_{i}^2\le \frac{c2^{\alpha-1}}{\tau_{n}^{\alpha-1}}.
\end{align*}
Using this, we can rewrite and further estimate~\eqref{ko5} on $\mathcal{E}$ as 
\begin{align*}
\frac{1}{\Theta_n^{\alpha-1}}-c(\alpha-1)
&\le (\alpha-1)\tau_n^{\alpha-1}\sum_{i=n}^{\infty}\xi_{i}\e_{i+1}\\
&\le 
(\alpha-1)\tau_n^{\alpha-1}T_n^{-\frac{\alpha}{2}}\sqrt{\frac{c2^{\alpha-1}}{\tau_{n}^{\alpha-1}}}
\Big[\frac{1}{\zeta_{n}\sqrt{A_{n}}}\sum_{i=n}^{\infty}\xi_{i}\e_{i+1}\Big]\\
&\le (\alpha-1)\sqrt{c 2^{\alpha-1}}\Theta_n^{-\frac{\alpha}{2}}\tau_n^{-\frac 1 2 }
\Big[\frac{1}{\zeta_{n}\sqrt{A_{n}}}\sum_{i=n}^{\infty}\xi_{i}\e_{i+1}\Big].
\end{align*}
%
%
%
By Lemma~\ref{l:liminf0} we obtain
\begin{align*}
\liminf_{n\to\infty}\Big[\frac{1}{\Theta_n^{\alpha-1}}- c(\alpha-1)\Big]\Theta_n^{\frac{\alpha}{2}}\sqrt{\tau_n}
<\infty.
\end{align*}
Since $\Theta_n\to 0$ we obtain 
\begin{align}
\label{c1}
\liminf\limits_{n\to\infty}\Theta_n^{1-\frac{\alpha}{2}}\sqrt{\tau_n}<\infty.
\end{align}
If $\alpha\ge 2$ we get a contradiction as $\Theta_n\le 1$. If $1<\alpha<2$ then~\eqref{c1} implies 
\begin{align*}
\liminf\limits_{n\to\infty}T_n\tau_n^{\frac{\alpha-1}{2-\alpha}}<\infty,
\end{align*}
and we get a contradiction as $T_n\ge 1$ and hence the value should be infinite.
\end{proof}

Now we turn to the case when $\rho_n\to\infty$, and the approximation~\eqref{aprox} is no longer valid. 
Instead, we will look at the iterations~\eqref{itit} in more detail. 

\begin{prop} 
\label{p:fast}
Suppose $\alpha>1$, $\theta=0$, {\rm (S)} is satisfied, $\rho_n\to\infty$, and $\lambda<1$. Then 
$\P(\mathcal{M})=1.$
\end{prop}

\begin{proof}
By symmetry and Theorem~\ref{main0}, without loss of generality it suffices to consider the event 
\begin{align*}
\mathcal{E}=\{\Theta_n\to 0\}\cap \{T_n\to\infty\}
\end{align*}
and prove that it has probability zero.
\smallskip




We begin by defining a small parameter $\delta$ in the following way.
Let $q\in (\lambda,1)$. 
Let $\e>0$ be small enough so that 
\begin{align}
\label{eeqq}
(\lambda+\e)(1+2\e)<q.
\end{align}
Let $\gamma\in(\max\{0,\alpha-2\},\alpha-1)$ and let $\delta>0$ be small enough so that 
\begin{align}
\label{ddd}
\delta 2^{\alpha-1}<\e
\end{align}
and 
\begin{align}
\label{ddd2}
2^{\frac{\alpha-1}{2}} \delta^{\frac{\gamma}{2(\alpha-1)}} 
\max_{m\in\N_0}\big[(m+1)q^{\frac{\gamma m}{2(\alpha-1)}}\big]<\e.
\end{align}

Further, we denote $\kappa_0=0$, and for each $n\in \N$, define the stopping times
\begin{align*}
\kappa_n=\inf\Big\{i>\kappa_{n-1}: \frac{\sigma_{i+1}}{\tau_i^{\alpha}}T_i^{\alpha}\le \delta T_i\Big\}
\end{align*}
to be the subsequent times when the second term in the iteration~\eqref{itit} is $\varepsilon$-smaller than the first term. 
\smallskip

The proof of the theorem now consist of the following four steps.  First, we show that on $\mathcal{E}$ the $\delta$-negligibility of the second term occurs infinitely often, that is, all $\kappa_n$ are finite. Then we provide an upper bound for $\e_i, i\ge \kappa_n$, as a function of both $n$ and $i$.  Using that bound, we observe that we can find a stopping time $\kappa_{\nu}$ such that for all subsequent times $n\ge \kappa_{\nu}$
the second term in~\eqref{itit} will be $\delta q^{n-\kappa_{\nu}}$-smaller than the first term. This much stronger domination of the first term finally allows us to show that $(T_n)$ is bounded on $\mathcal{E}$, which leads to a contradiction. 
\bigskip 

\emph{Step 1.} Let us prove that $\kappa_n<\infty$ almost surely for all $n$ on the event $\mathcal{E}$. 
Suppose this is not the case, that is, 
\begin{align*}
\bar n=\sup\Big\{i\in \N:\frac{\sigma_{i+1}}{\tau_i^{\alpha}}T_i^{\alpha}\le \delta T_i\Big\}
\end{align*}
is finite on $\mathcal{E}$ with positive probability. 
Roughly speaking, this means that in the iteration~\eqref{itit} the second term plays the main r\^ ole eventually, 
and (ignoring the other terms and constants for the moment) for all $n>k>\bar n$ we  have approximately
\begin{align*}
T_n
&\preceq T_{n-1}^{\alpha}\frac{\sigma_n}{\tau_{n-1}^{\alpha}}
\preceq T_k^{\alpha^{n-k}}
\frac{\sigma_{n}}{\tau_{n-1}^{\alpha}}
\Big(\frac{\sigma_{n-1}}{\tau_{n-2}^{\alpha}}\Big)^{\alpha}
\cdots \Big(\frac{\sigma_{k+1}}{\tau_k^{\alpha}}\Big)^{\alpha^{n-k-1}}
\le \tau_n\Theta_k^{\alpha^{n-k}},
\end{align*}
where $\preceq$ stands for ``approximately less''. Since $\theta=0$ this would lead to 
\begin{align*}
\limsup_{n\to\infty}\alpha^{-n}\log T_n \preceq \alpha^{-k}\log \Theta_k<0
\end{align*}
for sufficiently large $k$ since $\Theta_k\to 0$ on $\mathcal{E}$. This, however, would be a contradiction to $T_n\to 0$. 
\smallskip

Now we will make this argument precise. 
On the event $\{\bar n<\infty\}\cap \mathcal{E}$ by~\eqref{itit} we have for all
$n\ge \bar n$ 
\begin{align*}
T_{n+1}
&\le \big(\delta^{-1}+2^{\alpha-1}\big)T_{n}^{\alpha}\frac{\sigma_{n+1}}{\tau_{n}^{\alpha}}
+\e_{n+1}\sqrt{T_n^{\alpha}\frac{\sigma_{n+1}}{\tau_n^{\alpha}}(1-P_{n})}
\le cT_{n}^{\alpha}\frac{\sigma_{n+1}}{\tau_{n}^{\alpha}}(1
+\hat\xi_n\e_{n+1}).
\end{align*}
where $c=\delta^{-1}+2^{\alpha-1}$ and 
\begin{align*}
\hat \xi_n=
c^{-1}\Big[T_n^{\alpha}\frac{\sigma_{n+1}}{\tau_{n}^{\alpha}}\Big]^{-\frac{1}{2}}\sqrt{1-P_{n}}
=c^{-1}\Big[\max\Big\{T_n^{\alpha}\frac{\sigma_{n+1}}{\tau_{n}^{\alpha}},\delta T_n\Big\}\Big]^{-\frac{1}{2}}\sqrt{1-P_{n}}.
\end{align*}
Iterating and using $\sigma_i\le\tau_i$ and $\log(1+x)\le x$ we obtain for all $k>\bar n$ and all $n>k$
\begin{align*}
T_n
&\le c^{1+\alpha+\cdots+\alpha^{n-k-1}}T_k^{\alpha^{n-k}}
\frac{\sigma_{n}}{\tau_{n-1}^{\alpha}}
\Big(\frac{\sigma_{n-1}}{\tau_{n-2}^{\alpha}}\Big)^{\alpha}
\cdots \Big(\frac{\sigma_{k+1}}{\tau_k^{\alpha}}\Big)^{\alpha^{n-k-1}}
 \prod_{i=k}^{n-1} (1+\hat\xi_i\e_{i+1})^{\alpha^{n-i-1}}\\
&\le \big(c^{\frac{1}{\alpha-1}}\big)^{\alpha^{n-k}}\tau_n\Theta_k^{\alpha^{n-k}}
\prod_{i=k}^{n-1} (1+\hat\xi_i\e_{i+1})^{\alpha^{n-i-1}}\\
&\le \tau_n\Big[c^{\frac{1}{\alpha-1}}\Theta_k\exp\Big\{
\sum_{i=k}^{n-1} \alpha^{k-i-1}\log(1
+\hat\xi_i\e_{i+1}]\Big\}
\Big]^{\alpha^{n-k}}\\
&\le \tau_n\Big[c^{\frac{1}{\alpha-1}}\Theta_k\exp\Big\{\alpha^{k-1}
\sum_{i=k}^{n-1} \xi_i\e_{i+1}\Big\}
\Big]^{\alpha^{n-k}},
\end{align*}
where 
\begin{align*}
\xi_i=\alpha^{-i}\hat\xi_i.
\end{align*}

As $\theta=0$ this implies, for all $k\ge \bar n$, 
\begin{align}
\label{sub}
\limsup_{n\to\infty}\alpha^{-n}\log T_n
\le \alpha^{-k}\log \Big[c^{\frac{1}{\alpha-1}}\Theta_k\exp\Big\{\alpha^{k-1}
\sum_{i=k}^{\infty} \xi_i\e_{i+1}\Big\}
\Big].
\end{align}
It suffices to show that on the event $\mathcal{E}$ the expression on the right hand side is negative for some $k$ (depending on $\omega\in \mathcal{E}$), as it would imply that $T_n\to 0$ contradicting $T_n\ge 1$.
Since $\Theta_k\to 0$ on $\mathcal{E}$, it amounts to proving that 
\begin{align}
\label{sub9}
\liminf_{k\to\infty}\Big[\alpha^{k-1}
\sum_{i=k}^{\infty} \xi_i\e_{i+1}\Big]<\infty.
\end{align}

To show that this holds almost surely we use Lemma~\ref{l:liminf0}. 
Observe that since $T_i\ge 1$, on the whole probability space we have 
\begin{align*}
|\xi_i|\le \alpha^{-i}|\hat \xi_i|
\le \frac{\alpha^{-i}}{c\sqrt {\delta T_i}}
\le \frac{\alpha^{-i}}{c\sqrt {\delta}},
\end{align*}
for all $i\ge n$ and all $n$.
With 
\begin{align*}
A_n=\frac{1}{\delta c^2}\sum_{i=n}^{\infty}\alpha^{-2i}=
\frac{\alpha^{-2n+2}}{\delta c^2(\alpha^2-1)}
\end{align*}
we have by Lemma~\ref{l:liminf0}
\begin{align*}
\liminf_{k\to\infty}\Big[\alpha^{k-1}
\sum_{i=k}^{\infty} \xi_i\e_{i+1}\Big]
&=\frac{1}{\sqrt{\delta c^2(\alpha^2-1)}}\liminf_{k\to\infty}\Big[\frac{1}{\sqrt{A_{k}}}\sum_{i=k}^{\infty} \xi_i\e_{i+1}\Big]<\infty
\end{align*}
as required in~\eqref{sub9}.

\bigskip

\emph{Step 2.} Let $(c_n)$ be a real-valued sequence tending to infinity. 
For each $n$, consider the event 
\begin{align*}
\mathcal{E}_n=\big\{\kappa_n<\infty\text{ and }\e_{i}\le c_{i} (i-\kappa_n)\text{ for all }i>\kappa_n\big\}\cup\{\kappa_n=\infty\}.
\end{align*}
The aim of this step is to show that
\begin{align}
\label{io}
\P(\mathcal{E}_n\text{ i.o.})=1,
\end{align} 
for which it suffices to prove that
\begin{align}
\label{eqone}
\lim_{n\to\infty}\P(\mathcal{E}_n)=1.
\end{align} 

Using Chebychev's inequality, \eqref{norm}, and monotone convergence theorem we have 
\begin{align*}
\P(\mathcal{E}_n)
&=\lim_{j\to\infty}\P\big(\kappa_n<\infty\text{ and }\e_{i}\le c_{i} (i-\kappa_n)\text{ for all } \kappa_n< i\le \kappa_n+j\big)+\P(\kappa_n=\infty)\\
&\ge \lim_{j\to\infty}\E\Big[\one\{\kappa_n<\infty\}\prod_{i=\kappa_n+1}^{\kappa_n+j-1}\one\big\{\e_{i}\le c_{i} (i-\kappa_n)\big\}\cdot
\Big(1-\P_{\mathcal{F}_{\kappa_n+j-1}}\big(\e_{\kappa_n+j}^2> c_{\kappa_n+j}^2 j^2\big)\Big)\Big]\\
&\phantom{aaall}+\P(\kappa_n=\infty)\\
&\ge \lim_{j\to\infty}\E\Big[\one\{\kappa_n<\infty\}\prod_{i=\kappa_n+1}^{\kappa_n+j-1}\one\big\{\e_{i}\le c_{i} (i-\kappa_n)\big\}\cdot\Big(1-\frac{1}{c_{\kappa_n+j}^2j^2}\Big)\Big]+\P(\kappa_n=\infty)\ge \cdots\\
&\ge\lim_{j\to\infty}\E\Big[\one\{\kappa_n<\infty\}\prod_{i=\kappa_n+1}^{\kappa_n+j}\Big(1-\frac{1}{c_{i}^2(i-\kappa_n)^2}\Big)\Big]+\P(\kappa_n=\infty)\\
&=\E\Big[\one\{\kappa_n<\infty\}\exp\Big\{\sum_{i=1}^{\infty}\log\Big(1-\frac{1}{c_{\kappa_n+i}^2i^2}\Big)\Big\}\Big]+\P(\kappa_n=\infty).
\end{align*}
Let $\hat c_n= \min \{c_i:i\ge n\}$ and observe that $\hat c_n\to\infty$. Hence for all sufficiently large $n$ we can estimate 
\begin{align*}
\log(1-x)\ge-x-x^2
\end{align*}
for all $x\in [0,\hat c_n^{-2}]$.
Since $\kappa_n+i\ge n$ for all $n$ and $i$ we have $c_{\kappa_n+i}\ge \hat c_n$. Hence 
for all sufficiently large $n$ we obtain 
\begin{align*}
\P(\mathcal{E}_n)
&\ge \E\Big[\one\{\kappa_n<\infty\}\exp\Big\{-\frac{1}{\hat c_{n}^2}\sum_{i=1}^{\infty}\frac{1}{i^2}-\frac{1}{\hat c_{n}^4}\sum_{i=1}^{\infty}\frac{1}{i^4}\Big\}\Big]+\P(\kappa_n=\infty)\to 1
\end{align*}
as $n\to\infty$, implying~\eqref{eqone}.
\bigskip

\emph{Step 3.} 
Let us show that on the event $\mathcal{E}$ there is $\nu$ (depending on $\omega\in \mathcal{E}$) such that for all $n\ge \kappa_{\nu}$
\begin{align}
\label{easy}
\e_n\le n
\end{align}
and
\begin{align}
\label{ind1}
\frac{\sigma_{n+1}}{\tau_{n}^{\alpha}}T_{n}^{\alpha-1}\le \delta q^{n-\kappa_{\nu}}.
\end{align}

Roughly speaking, we know that the first term of~\eqref{itit} dominates at time $\kappa_{\nu}$. Ignoring the remaining terms and taking into account~\eqref{st} we can write 
for $n>\kappa_{\nu}$ approximately
\begin{align*}
\frac{\sigma_{n+1}}{\tau_n^{\alpha}}T_n^{\alpha-1}\preceq \frac{\sigma_{n+1}}{\tau_n^{\alpha}}T_{n-1}^{\alpha-1}
=\frac{\sigma_{n+1}\tau_{n-1}^{\alpha}}{\sigma_n\tau_n^{\alpha}}\cdot \frac{\sigma_{n}}{\tau_{n-1}^{\alpha}}T_n^{\alpha-1}
\approx \lambda\cdot \frac{\sigma_{n}}{\tau_{n-1}^{\alpha}}T_n^{\alpha-1},
\end{align*}
which demonstrates eventual exponential decay of these terms as pinpointed in~\eqref{ind1}. 
\smallskip

To make this rigorous and prove the existence of $\nu$ satisfying~\eqref{easy} and~\eqref{ind1}, 
we define a sequence
\begin{align*}
c_n=\Big[\frac{\sigma_n}{\tau_{n-1}^{\alpha}}\Big]^{\frac{\gamma}{2(\alpha-1)}-\frac 1 2}, \quad n\in\N,
\end{align*}
and observe that $c_n\to \infty$ by
Lemma~\ref{omegazero} and since $\gamma<\alpha-1$. Further, it follows from~\eqref{d:lambdan} and~\eqref{st} that 
\begin{align}
\label{ls}
\limsup_{n\to\infty}\frac{\sigma_{n+1}\tau_{n-1}^{\alpha}}{\sigma_{n}\tau_{n}^{\alpha}}
=\limsup_{n\to\infty}\lambda_n=\lambda<1.
\end{align}

According to Steps 1 and 2, by Lemma~\ref{le1}, and by~\eqref{ls}  we can pick $\nu$ so that $\kappa_{\nu}<\infty$,  $\mathcal{E}_{\nu}$ occurs, \eqref{easy}  holds, and for all $n>\kappa_{\nu}$
\begin{align}
\label{leps}
\frac{\sigma_{n+1}\tau_{n-1}^{\alpha}}{\sigma_{n}\tau_{n}^{\alpha}}<\lambda+\e.
\end{align}

Let us now prove~\eqref{ind1} by induction in $n$. It is true for $n=\kappa_{\nu}$ by definition of $\kappa_{\nu}$. 
For $n>\kappa_{\nu}$ using~\eqref{itit}, \eqref{psibounds},  and the fact that $\mathcal{E}_{\nu}$ occurs we obtain 
\begin{align}
\frac{\sigma_{n+1}}{\tau_n^{\alpha}}T_n^{\alpha-1}
&\le \frac{\sigma_{n+1}}{\tau_n^{\alpha}}\Big(
T_{n-1}+2^{\alpha-1}T_{n-1}^{\alpha}\frac{\sigma_{n}}{\tau_{n-1}^{\alpha}} 
+\e_{n}\sqrt{\sigma_{n}P_{n-1}(1-P_{n-1})}
\Big)^{\alpha-1}\notag\\
&\le \frac{\sigma_{n+1}}{\tau_n^{\alpha}}
\Big(T_{n-1}+2^{\alpha-1}T_{n-1}^{\alpha}\frac{\sigma_n}{\tau_{n-1}^{\alpha}}
+c_n (n-\kappa_{\nu})\sqrt{2^{\alpha-1}T_{n-1}^{\alpha}\frac{\sigma_n}{\tau_{n-1}^{\alpha}}}\Big)^{\alpha-1}\notag\\
&\le \frac{\sigma_{n+1}}{\tau_n^{\alpha}}T_{n-1}^{\alpha-1}
\Big(1+2^{\alpha-1}T_{n-1}^{\alpha-1}\frac{\sigma_n}{\tau_{n-1}^{\alpha}}
+c_n (n-\kappa_{\nu})\sqrt{2^{\alpha-1}T_{n-1}^{\alpha-2}\frac{\sigma_n}{\tau_{n-1}^{\alpha}}}\Big)^{\alpha-1}.
\label{eps3}
\end{align}

Using~\eqref{ind1} for $n-1$ we have 
\begin{align}
\label{eps12}
\frac{\sigma_{n+1}}{\tau_{n}^{\alpha}}T_{n-1}^{\alpha-1}
=\delta q^{n-\kappa_{\nu}-1}\frac{\sigma_{n+1}\tau_{n-1}^{\alpha}}{\tau_{n}^{\alpha}\sigma_n}.
\end{align}

Using~\eqref{ind1} for $n-1$, $q<1$, and~\eqref{ddd} we have 
\begin{align}
\label{eps1}
2^{\alpha-1}T_{n-1}^{\alpha-1}\frac{\sigma_n}{\tau_{n-1}^{\alpha}}
\le 2^{\alpha-1}\delta q^{n-\kappa_{\nu}-1}\le \delta 2^{\alpha-1}<\e.
\end{align}

Using~\eqref{ind1} for $n-1$, $T_{n-1}\ge 1$, and $\gamma>\alpha-2$ we estimate 
\begin{align*}
T_{n-1}^{\alpha-2}\frac{\sigma_n}{\tau_{n-1}^{\alpha}}
<T_{n-1}^{\gamma}\frac{\sigma_n}{\tau_{n-1}^{\alpha}}
\le (\delta q^{n-\kappa_{\nu}-1})^{\frac{\gamma}{\alpha-1}} \Big[\frac{\sigma_n}{\tau_{n-1}^{\alpha}}\Big]^{1-\frac{\gamma}{\alpha-1}}
\end{align*}
and obtain by~\eqref{ddd2} and by the choice of $(c_n)$
\begin{align}
\label{eps2}
c_n(n-\kappa_{\nu})\sqrt{2^{\alpha-1}T_{n-1}^{\alpha-2}\frac{\sigma_n}{\tau_{n-1}^{\alpha}}}
&\le 2^{\frac{\alpha-1}{2}}(n-\kappa_{\nu})
(\delta q^{n-\kappa_{\nu}-1})^{\frac{\gamma}{2(\alpha-1)}} <\e.
\end{align}
Substituting~\eqref{eps12}, \eqref{eps1}, and~\eqref{eps2} into~\eqref{eps3} as well as using~\eqref{leps} and~\eqref{eeqq} we obtain 
\begin{align*}
\frac{\sigma_{n+1}}{\tau_n^{\alpha}}T_n^{\alpha-1}
&\le \delta q^{n-\kappa_{\nu}-1}\frac{\sigma_{n+1}\tau_{n-1}^{\alpha}}{\tau_n^{\alpha}\sigma_n}(1+2\e)
< \delta q^{n-\kappa_{\nu}-1}(\lambda+\e)(1+2\e)
<\delta q^{n-\kappa_{\nu}},
\end{align*}
completing the induction.
\bigskip

\emph{Step 4.} Finally, let us prove that $\P(\mathcal{E})=0$. 
For all $n\ge \kappa_{\nu}$ we have on the event $\mathcal{E}$ using~\eqref{itit}, \eqref{psibounds}, \eqref{easy},~\eqref{ind1}, and $T_{n-1}\ge 1$
\begin{align*}
T_n
&\le T_{n-1}\Big(1+2^{\alpha-1}T_{n-1}^{\alpha-1}\frac{\sigma_{n}}{\tau_{n-1}^{\alpha}}+n T_{n-1}^{-\frac 1 2}\sqrt{2^{\alpha-1}T_{n-1}^{\alpha-1}\frac{\sigma_{n}}{\tau_{n-1}^{\alpha}}}\Big)\\
&\le T_{n-1}\big(1+2^{\alpha-1}\delta q^{n-\kappa_{\nu}-1}+n\sqrt{2^{\alpha-1}\delta q^{n-\kappa_{\nu}-1}}\big)\le \cdots\\
&\le T_{\kappa_{\nu}}\prod_{i=1}^{n-\kappa_{\nu}}\Big(1+2^{\alpha-1}\delta q^{i-1}+(\kappa_{\nu}+i)\sqrt{2^{\alpha-1}\delta q^{i-1}}\Big)\\
&\le T_{\kappa_{\nu}}\prod_{i=1}^{\infty}\Big(1+2^{\alpha-1}\delta q^{i-1}+(\kappa_{\nu}+i)\sqrt{2^{\alpha-1}\delta q^{i-1}}\Big).
\end{align*}
As the infinite product converges, this implies that $T_n$ is bounded thus contradicting to the definition of the event $\mathcal{E}$. 
\end{proof}

\bigskip

\begin{proof}[Proof of Theorem~\ref{main1}]
If $(\rho_n)$ is bounded the statement follows from Proposition~\ref{p:mod1}. If $\rho_n\to\infty$ and $\lambda<1$ the result follows from Proposition~\ref{p:fast}. If $\rho_n\to\infty$ and $\lambda>1$
it suffices to apply Lemmas~\ref{omegazero} and~\ref{diverges}.
Finally, the last statement follows from Lemma~\ref{l:regular}.
\end{proof}
\smallskip



\section{Critical regime}

\label{s:cri}

The aim of this section is to prove Theorem~\ref{main3}.
For all $n$, denote 
\begin{align*}
\phi_n=\tau_n e^{-\theta \alpha^n}.
\end{align*}
It is easy to see that 
\begin{align}
\label{sec}
\lim_{n\to\infty}\alpha^{-n}\log \phi_n=0
\end{align}
and 
\begin{align}
\label{imply}
\frac{\tau_{n+1}}{\tau_n^{\alpha}}=\frac{\phi_{n+1}}{\phi_n^{\alpha}}, \quad n\in \N_0
\qquad\text{implying}\qquad
\sum_{n=0}^{\infty}\frac{\tau_{n+1}}{\tau_n^{\alpha}}<\infty
\quad\Leftrightarrow\quad
\sum_{n=0}^{\infty}\frac{\tau_{n+1}}{\tau_n^{\alpha}}<\infty.
\end{align}
In the case when the above series diverge we will show that no monopoly occurs using the same method as in Section~\ref{s:nomo}. 
Otherwise, if the series converge, we will explicitly construct two non-trivial events such that on one of them 
we have monopoly and on the other one we don't. 

\begin{lemma} 
\label{l:lambda}
If 
\begin{align*}
\sum_{n=0}^{\infty}\frac{\phi_{n+1}}{\phi_n^{\alpha}}<\infty
\end{align*}
then $(\phi_n)$ is unbounded.
\end{lemma}

\begin{proof} Suppose $(\phi_n)$ is bounded by some $c<\infty$. 
Let $\e>0$ be such that $c\e^{\frac{1}{\alpha-1}}<1$ and let
$k\in\N_0$ be such that for all $n\ge k$
\begin{align*}
\phi_{n+1}\le \e\phi_n^{\alpha},
\end{align*}
which is possible since the series converges. 
Iterating this inequality we obtain for all $n\ge k$
\begin{align*}
\phi_n\le \e^{1+\alpha+\cdots+\alpha^{n-k-1}}\phi_k^{\alpha^{n-k}}=\e^{\frac{\alpha^{n-k}-1}{\alpha-1}}\phi_k^{\alpha^{n-k}}.
\end{align*}
This implies 
\begin{align*}
\lim_{n\to\infty}\alpha^{-n}\log\phi_n
=\alpha^{-k}\log\big(\phi_k\e^{\frac{1}{\alpha-1}}\big)
\le\alpha^{-k}\log\big(c\e^{\frac{1}{\alpha-1}}\big)<0
\end{align*}
contradicting~\eqref{sec}. 
\end{proof}

\begin{prop} 
\label{p:cri1}
Suppose $\alpha>1$ and $\theta\in (0,\infty)$.  
Then $\P(\mathcal{M})<1$. Further, if 
\begin{align}
\label{s8}
\sum_{n=0}^{\infty}\frac{\phi_{n+1}}{\phi_n^{\alpha}}=\infty
\end{align}
then $\P(\mathcal{M})=0$. 
%
\end{prop}

\begin{proof} 
First, suppose that the series~\eqref{s8} diverges. By~\eqref{sec} we have 
\begin{align*}
\frac{\sigma_{n+1}}{\tau_n^{\alpha}}
=\frac{\tau_{n+1}-\tau_n}{\tau_n}
=\frac{\phi_{n+1}-\phi_n e^{-\theta(\alpha-1)\alpha^n}}{\phi_n^{\alpha}}\sim\frac{\phi_{n+1}}{\phi_n^{\alpha}}
\end{align*}
as $n\to\infty$. Hence by~\eqref{s8} we have
\begin{align}
\label{div}
\sum_{n=0}^{\infty}\frac{\sigma_{n+1}}{\tau_n^{\alpha}}=\infty
\end{align}
and $\P(\mathcal{M})=0$ by Lemma~\ref{diverges}.
\medskip

It remains to consider the case when the series on the left-hand side of~\eqref{s8} converges and show that $\P(\mathcal{M})<1$. 
To do so we will construct an event $\mathcal{E}$ such that $\P(\mathcal{E})>0$ and both $T_n$ and $\hat T_n=\tau_n-T_n$ tend to infinity on 
$\mathcal{E}$.
\smallskip

Let $\gamma>\frac{2}{\alpha-1}$ and, for all $n\in\N$, denote 
\begin{align}
\label{defchi}
\chi_n=\max_{1\le k\le n} k^{\gamma}\phi_k.
\end{align}
By Lemma~\ref{l:lambda} $\chi_n\to\infty$. 
Let $\beta\in (\frac{2+\gamma}{\alpha},\gamma)$. 
Choose $m\in\N$
large enough so that
\begin{align}
\label{he10}
\tau_m-2\chi_m>1,
\end{align}
\begin{align}
\label{he9}
T_0\le \tau_m-\chi_m
\qquad\text{and}\qquad
T_0+\tau_m-\tau_0\ge \chi_m,
\end{align}
as well as  
\begin{align}
\label{he1}
\big(\phi_{n+1}-\phi_ne^{-\theta(\alpha-1)\alpha^{n}}\big)n^{\alpha\gamma}&> \phi_{n+1}n^{\alpha\beta},\\
\label{he1b}
\chi_12^{-\gamma} n^{\alpha\beta-\gamma}&>n^2,\\
\label{he6}
n^{\alpha\beta}\Big(1-\Big[\frac{2^{\gamma}n^{2-\alpha\beta+\gamma}}{\chi_1}\Big]^{1/2}\Big)
&> (n+1)^{\gamma},
\end{align}
for all $n\ge m$. Indeed, \eqref{he10} and \eqref{he9} are possible by~\eqref{sec} since $(\tau_n)$ grows faster than $(\chi_n)$, \eqref{he1} is possible since $\gamma>\beta$ and again  by~\eqref{sec}, \eqref{he1b} is possible since $\alpha\beta-\gamma>2$, and~\eqref{he6} is possible as $2-\alpha\beta+\gamma<0$ and $\alpha\beta>\gamma$. 
Observe that the event
\begin{align*}
\mathcal{E}_0=\big\{T_m\in [\chi_m, \tau_m-\chi_m]\big\} 
\end{align*}
occurs with positive probability by~\eqref{he10} and~\eqref{he9} since $T_m$ takes each value between $T_0$ and $T_0+\tau_m-\tau_0$ with positive probability. 
Let
\begin{align*}
\mathcal{E}=\mathcal{E}_0\cap\big\{|\e_{n+1}|\le n\text{ for all }n\ge m\big\}.
\end{align*}
By Chebychev's inequality and using~\eqref{norm} we obtain 
\begin{align*}
\P(\mathcal{E})
&=\lim_{n\to\infty}\E\Big[\one\{\mathcal{E}_0\}\one\big\{|\e_{i+1}|\le i\text{ for all }m\le i<n\big\}\big(1-\P_{\mathcal{F}_{n}}(\e_{n+1}^2>n^2)\big)\Big]\notag\\
&\ge \lim_{n\to\infty}\E\Big[\one\{\mathcal{E}_0\}\one\big\{|\e_{i+1}|\le i\text{ for all }1\le i< n\big\}\Big(1-\frac{1}{n^2}\Big)\Big]
\ge\cdots\notag\\
&\ge \P(\mathcal{E}_0)\prod_{n=m}^{\infty}\Big(1-\frac{1}{n^2}\Big)>0
\end{align*}
since the infinite product converges.
\smallskip

Due to the symmetry of the relations
\begin{align}
\label{he4}
T_{n+1}=T_n+\sigma_{n+1}\psi(\Theta_n)+\e_{n+1}\sqrt{\sigma_{n+1}\psi(\Theta_n)(1-\psi(\Theta_n))}
\end{align}
and
\begin{align*}
\hat T_{n+1}=\hat T_n+\sigma_{n+1}\psi(\hat \Theta_n)+\hat \e_{n+1}\sqrt{\sigma_{n+1}\psi(\hat \Theta_n)(1-\psi(\hat\Theta_n))},
\end{align*}
where $\hat\Theta_n=\hat T_n/\tau_n$ and $\hat \e_n=-\e_n$, and since $\hat T_m$ satisfies the condition 
$\hat T_m\in [\chi_m, \tau_m-\chi_m]$ in the same way as $T_m$
on $\mathcal{E}_0$, we only need to prove that $T_n\to\infty$ on $\mathcal{E}$. 
\smallskip

Since $\chi_n\to\infty$, it suffices to show that $T_n\ge \chi_n$ on $\mathcal{E}$ for all $n\ge m$, which we prove by induction. 
For $n=m$ this follows from the definition of $\mathcal{E}_0$. 
Let $n\ge m$. If $\chi_{n+1}=\chi_n$ then we have $T_{n+1}\ge T_n\ge \chi_n=\chi_{n+1}$ as required. 
If $\chi_{n+1}>\chi_n$ then 
\begin{align}
\label{222}
\phi_{n+1}>\frac{\chi_n}{(n+1)^\gamma}\ge \frac{\chi_1}{2^{\gamma}n^{\gamma}}.
\end{align}
It follows from~\eqref{he4} that on $\mathcal{E}$
\begin{align}
T_{n+1}
&\ge \sigma_{n+1}\psi(\Theta_n) 
-n\sqrt{\sigma_{n+1}\psi(\Theta_n)}.
\label{he2}
\end{align}
Using~\eqref{psibounds}, the induction hypothesis $T_n\ge \chi_n$, $\chi_n\ge n^{\gamma}\phi_n$
following from the definition~\eqref{defchi}, and~\eqref{he1} we obtain 
\begin{align}
\sigma_{n+1}\psi(\Theta_n)
&\ge \sigma_{n+1}\Theta_n^{\alpha}
= \frac{\tau_{n+1}-\tau_n}{\tau_n^{\alpha}}\,T_n^{\alpha}
\ge\frac{\phi_{n+1}e^{\theta\alpha^{n+1}}-\phi_ne^{\theta\alpha^{n}}}{\phi_n^{\alpha}e^{\theta\alpha^{n+1}}}\chi_n^{\alpha}\notag\\
&\ge \big(\phi_{n+1}-\phi_ne^{-\theta(\alpha-1)\alpha^{n}}\big)n^{\alpha\gamma}>\phi_{n+1}n^{\alpha\beta}.
\label{he3}
\end{align}
Since the function $x\mapsto x-n\sqrt x$ is increasing on $[n^2,\infty)$ 
\begin{align*}
\phi_{n+1}n^{\alpha\beta}\ge \chi_12^{-\gamma} n^{\alpha\beta-\gamma}>n^2
\end{align*}
by~\eqref{222} and \eqref{he1b}, we obtain by~\eqref{he2}, \eqref{he3}, \eqref{222}, and~\eqref{he6} that 
\begin{align*}
T_{n+1}
&\ge n^{\alpha\beta}\phi_{n+1}-n^{\frac{\alpha\beta}{2}+1}\sqrt{\phi_{n+1}}
\ge n^{\alpha\beta}\phi_{n+1}\Big(1-\Big[\frac{n^{2-\alpha\beta}}{\phi_{n+1}}\Big]^{1/2}\Big)\\
&\ge n^{\alpha\beta}\phi_{n+1}\Big(1-\Big[\frac{2^{\gamma}n^{2-\alpha\beta+\gamma}}{\chi_1 }\Big]^{\frac 1 2}\Big)
\ge (n+1)^{\gamma}\phi_{n+1}=\chi_{n+1}
\end{align*}
as required. 
\end{proof}

\begin{prop} 
\label{p:cri2}
Suppose $\alpha>1$ and $\theta\in (0,\infty)$. If
\begin{align}
\label{coco}
\sum_{n=0}^{\infty}\frac{\phi_{n+1}}{\phi_n^{\alpha}}<\infty
\end{align}
then $\P(\mathcal{M})>0$.
\end{prop}

\begin{proof} 
For all $n\in\N_0$ it follows from~\eqref{psibounds}, \eqref{iteration1}, and $\sigma_{n+1}\le \tau_{n+1}$ that 
\begin{align}
T_{n+1}
&\le T_n+\tau_{n+1}2^{\alpha-1}\Theta_n^{\alpha}+\e_{n+1}\sqrt{\sigma_{n+1}P_n(1-P_n)}\notag\\
&\le T_n+2^{\alpha-1}T_n^{\alpha}\frac{\phi_{n+1}}{\phi_n^{\alpha}} 
+\e_{n+1}\sqrt{\sigma_{n+1}P_n(1-P_n)},
\label{itit7}
\end{align}
implying 
\begin{align}
\label{fdif17}
\frac{T_{n+1}-T_n}{T_n^{\alpha}}\le 2^{\alpha-1}\frac{\phi_{n+1}}{\phi_n^{\alpha}}+\xi_{n}\e_{n+1},
\end{align}
where 
\begin{align*}
\xi_{n}=T_n^{-\alpha}\sqrt{\sigma_{n+1}P_n(1-P_n)}.
\end{align*}

As the series~\eqref{coco} converges we can pick $m$ large enough so that 
\begin{align}
\label{m8}
2^{\alpha-1}\sum_{n=m}^{\infty}\frac{\phi_{n+1}}{\phi_n^{\alpha}}<\frac{1}{\alpha-1}\cdot \frac{1}{T_0^{\alpha-1}}.
\end{align}
Consider the event 
\begin{align*}
\mathcal{E}=\Big\{2^{\alpha-1}\sum_{n=m}^{\infty}\frac{\phi_{n+1}}{\phi_n^{\alpha}}<\frac{1}{\alpha-1}\cdot\frac{1}{T_m^{\alpha-1}}\text{ \rm and }
\sum_{n=m}^{\infty}\xi_{n}\e_{n+1} \le 0\Big\}.
\end{align*}
We have $\P(\mathcal{E})>0$ by~\eqref{m8}, as the first condition can be achieved by $B_1=\cdots=B_m=0$ occurring with positive probability, and since the conditional expectation of the second series given $\mathcal{F}_m$ is zero almost surely. Let us show that 
$\mathcal{E}\subset \mathcal{M}$.
Suppose $T_n\to\infty$ for some $\omega\in \mathcal{E}$. 
We have by~\eqref{fdif17} and the second condition defining the event $\mathcal{E}$
\begin{align*}
2^{\alpha-1}\sum_{n=m}^{\infty}\frac{\phi_{n+1}}{\phi_n^{\alpha}}\ge \sum_{n=m}^{\infty}\frac{T_{n+1}-T_n}{T_n^{\alpha}}
\ge \int_{T_m}^{\infty}\frac{dx}{x^{\alpha}}=\frac{1}{\alpha-1}\cdot\frac{1}{T_m^{\alpha-1}}
\end{align*}
leading to a contradiction with the first condition defining the event $\mathcal{E}$.
\end{proof}
\bigskip

\begin{proof}[Proof of Theorem~\ref{main3}]  
The statement follows from~\eqref{imply} and Propositions~\ref{p:cri1} and~\ref{p:cri2}.
\end{proof}

\end{document}